\documentclass{amsart}
\usepackage{amsmath, amssymb, amscd, amsthm, amsfonts, amstext, amsbsy,  mathrsfs, hyperref,  upgreek, mathtools, stmaryrd, enumitem, tensor, comment, lineno, tikz-cd}
\hypersetup{colorlinks=true}
\usepackage[shadow, textsize=footnotesize, color=blue!40
, bordercolor=blue!80]{todonotes}
\numberwithin{equation}{section}

\renewcommand{\subset}{\subseteq}

\newcommand{\B}{\mathcal{B}}
\newcommand{\C}{\mathcal{C}}
\newcommand{\D}{\mathcal{D}}

\newcommand{\F}{\mathcal{F}}

\newcommand{\M}{\mathcal{M}}

\renewcommand{\P}{\mathcal{P}}

\newcommand{\CC}{\mathbb{C}}

\newcommand{\GG}{\mathbb{G}}

\newcommand{\QQ}{\mathbb{Q}}
\newcommand{\RR}{\mathbb{R}}

\newcommand{\metr}{\mathscr{M}_\lambda}
\newcommand{\metrhat}{\widehat{\mathscr{M}}_\lambda}

\newcommand{\tr}{\mathcal{T}}
\newcommand{\pre}[2]{\tensor[^{#1}]{#2}{}}

\newcommand{\Nbhd}{\boldsymbol{N}}
\newcommand{\Bor}{\l^+\text{-}\mathbf{Bor}}

\newcommand{\markdef}[1]{\textbf{#1}}

\newcommand{\Sii}[2]{\l^+\text{-}\boldsymbol{\Sigma}^{ #1 }_{ #2}}
\newcommand{\Pii}[2]{\l^+\text{-}\boldsymbol{\Pi}^{ #1 }_{ #2 }}
\newcommand{\Dee}[2]{\l^+\text{-}\boldsymbol{\Delta}^{ #1 }_{ #2}}
\newcommand{\Dlim}{D\text{-}\!\lim}
\newcommand{\DDlim}{\mathcal{D}\text{-}\!\lim}
\newcommand{\llim}[1]{#1\text{-}\!\lim}
\DeclareMathOperator{\ulim}{ulim}
\DeclareMathOperator{\supp}{supp}
\newcommand{\ullim}[1]{#1\text{-}\!\ulim}

\newcommand{\llimhatdef}[2]{#1\text{-}\widehat{\lim}{}^{#2}}
\newcommand{\Fin}{\mathrm{Fin}}
\DeclareMathOperator{\lev}{lev}

\renewcommand{\a}{\alpha}
\renewcommand{\b}{\beta}
\newcommand{\g}{\gamma}

\renewcommand{\k}{\kappa}
\renewcommand{\l}{\lambda}

\renewcommand{\phi}{\varphi}
\renewcommand{\o}{\omega}

\newcommand{\cof}{\operatorname{cof}}

\newcommand{\diam}{\operatorname{diam}}

\newcommand{\pow}{\mathscr{P}}

\newcommand{\leng}{\operatorname{lh}}

\DeclareMathOperator{\conc}{\mathbin{ {}^\smallfrown {} }}

\newcommand{\ZFC}{{\sf ZFC}}

\DeclareMathOperator{\cl}{cl}
\DeclareMathOperator{\ran}{ran}
\DeclareMathOperator{\dom}{dom}

\newtheorem{theorem}{Theorem}[section]
\newtheorem{lemma}[theorem]{Lemma}
\newtheorem{corollary}[theorem]{Corollary}
\newtheorem{proposition}[theorem]{Proposition}

\theoremstyle{definition}
\newtheorem{claim}{Claim}[theorem]
\newtheorem*{claim*}{Claim}
\newtheorem{definition}[theorem]{Definition}

\newtheorem{question}[theorem]{Question}
\newtheorem*{question*}{Question}
\theoremstyle{remark}
\newtheorem{remark}[theorem]{Remark}

\newenvironment{enumerate-(a)}{\begin{enumerate}[label={\upshape (\alph*)}, leftmargin=2pc]}{\end{enumerate}}

\newenvironment{enumerate-(a)-r}{\begin{enumerate}[label={\upshape (\alph*)}, leftmargin=2pc,resume]}{\end{enumerate}}

\newenvironment{enumerate-(A)}{\begin{enumerate}[label={\upshape (\Alph*)}, leftmargin=2pc]}{\end{enumerate}}

\newenvironment{enumerate-(A)-r}{\begin{enumerate}[label={\upshape (\Alph*)}, leftmargin=2pc,resume]}{\end{enumerate}}

\newenvironment{enumerate-(i)}{\begin{enumerate}[label={\upshape (\roman*)}, leftmargin=2pc]}{\end{enumerate}}

\newenvironment{enumerate-(i)-r}{\begin{enumerate}[label={\upshape (\roman*)}, leftmargin=2pc,resume]}{\end{enumerate}}

\newenvironment{enumerate-(I)}{\begin{enumerate}[label={\upshape (\Roman*)}, leftmargin=2pc]}{\end{enumerate}}

\newenvironment{enumerate-(I)-r}{\begin{enumerate}[label={\upshape (\Roman*)}, leftmargin=2pc,resume]}{\end{enumerate}}

\newenvironment{enumerate-(1)}{\begin{enumerate}[label={\upshape (\arabic*)}, leftmargin=2pc]}{\end{enumerate}}

\newenvironment{enumerate-(1)-r}{\begin{enumerate}[label={\upshape (\arabic*)}, leftmargin=2pc,resume]}{\end{enumerate}}

\newenvironment{itemizenew}{\begin{itemize}[leftmargin=2pc]}{\end{itemize}}

\begin{document}

\title{Generalized Baire Class functions}
\date{\today}

\author{Luca Motto Ros}
\address[Luca Motto Ros]
{Universit\`a degli Studi di Torino,
Dipartimento di Matematica ``G. Peano'',
Via Carlo Alberto 10, 10123 Torino, Italy}
\email[Luca Motto Ros]{luca.mottoros@unito.it}

\author{Beatrice Pitton}
\address[Beatrice Pitton]{
Université de Lausanne,
Quartier UNIL-Chamberonne,
Bâtiment Anthropole,
1015 Lausanne, Switzerland and Universit\`a degli Studi di Torino,
Dipartimento di Matematica ``G. Peano'',
Via Carlo Alberto 10, 10123 Torino, Italy}
\email[Beatrice Pitton]{beatrice.pitton@unil.ch}

\thanks{Research partially supported by the project PRIN 2022 ``Models, sets and classifications'', prot.\ 2022TECZJA. The authors are members of GNSAGA (INdAM)}

\begin{abstract}
Let \( \l \) be an uncountable cardinal such that \( 2^{< \l} = \l \).
Working in the setup of generalized descriptive set theory, we study the structure of \( \l^+ \)-Borel measurable functions with respect to various kinds of limits, and isolate a suitable notion of \( \l \)-Baire class \( \xi \) function. Among other results, we provide higher analogues of two classical theorems of Lebesgue, Hausdorff, and Banach, namely:
\begin{enumerate-(1)}
\item 
A function is \( \l^+ \)-Borel measurable if and only if it can be obtained from continuous functions by iteratively applying pointwise \( D \)-limits, where \( D \) varies among directed sets of size at most \( \l \).
\item 
A function is of \( \l \)-Baire class \( \xi \) if and only if it is \( \Sii{0}{\xi+1} \)-measurable.
\end{enumerate-(1)}
\end{abstract}

\maketitle

\setcounter{tocdepth}{1}
\tableofcontents

\section{Introduction}
Roughly speaking, generalized descriptive set theory is the higher analogue of classical descriptive set theory obtained by replacing all occurrences of the first infinite cardinal \( \omega \) with an uncountable cardinal \( \l \) or its cofinality \( \mu = \cof(\l) \). For example, the generalized Cantor space \( \pre{\l}{2} \) is obtained by endowing the set of all binary sequences of length \( \l \) with the so-called bounded topology, while the generalized Baire space \( \pre{\mu}{\l}\) consists of all \( \mu \)-sequences with values in \( \l \), again equipped with the bounded topology (see Section~\ref{subsec:treesandspacesofsequences}). Because of some striking applications and tight connections with other well-established areas of mathematical logic, such as Shelah's stability in model theory~\cite{FHK14,HM17,HKM17,MMR21,M22,M2023}, generalized descriptive set theory has gained a certain relevance in modern set theory, and the quest for a solid foundation, paving the way to more applications, became an important issue.

Nowadays the literature features a thorough study of the classes of Polish-like spaces that allow a meaningful development of the theory~\cite{MR13, AMRS23, AST18, Gal, DW, phdago, DMR}, as well as a deep analysis of their definable subsets~\cite{FHK14, HK18,  LS15, LMRS16, DMR, AMR22, ACMRP, Ana}. 
The goal of this paper is instead to study definable functions between such spaces, focusing in particular on \( \l^+ \)-Borel measurable functions and their stratifications.%
\footnote{After completing this work, we were informed that the same kind of problems (but restricted to regular cardinals) were tackled in~\cite{phdNob} using completely different methods.
Unfortunately, the proof of~\cite[Theorem 4.12]{phdNob}, which is the main result in this direction from that source, is flawed, and the definition of \( \l \)-Baire class \( \xi \) functions given there cannot work as expected. We will come back to this issue in Section~\ref{sec:questions}.} 

In classical descriptive set theory, two of the most fundamental results concerning Borel functions between separable metrizable spaces, due to Lebesgue, Hausdorff, and Banach, are the following ones.

\begin{theorem}[{See e.g.\ \cite[Theorem 11.6, or Theorems 24.3 and 24.10]{Kech}}] \label{thm:closureunderlimits-classical}
Let \( X \) and \( Y \) be separable metrizable spaces, and further assume that either \( X \) is zero-dimensional or
\( Y = \RR \).
Then the class of Borel functions from \( X \) to \( Y \) coincides with the closure under pointwise limits of the class of continuous functions.
\end{theorem}

This can be refined by considering the Baire hierarchy, which is recursively defined by stipulating that a function \( f \colon X \to Y \) is of Baire class \( 0 \) if it is continuous, while it is of Baire class \( \xi > 0 \) if it can be written as a pointwise limit of functions of lower Baire classes.

\begin{theorem}[{See e.g.\ \cite[Theorems 24.3 and 24.10]{Kech}}] \label{thm:Bairestratification-classical}
Let \( X \) and \( Y \) be separable metrizable spaces, and further assume that either \( X \) is zero-dimensional or
\( Y = \RR \). Let \( \xi < \omega_1 \). Then \( f \) is a Baire class \( \xi \) function if and only if it is \( \boldsymbol{\Sigma}^0_{\xi+1} \)-measurable.
\end{theorem}

Notice that the extra hypotheses%
\footnote{Slight variations are possible. For example, one could let \( Y \) be an interval in \( \RR \), or \( \RR^n \), or \( \CC^n \), and so on.}
on the spaces \( X \) and \( Y \) in Theorems~\ref{thm:closureunderlimits-classical} and~\ref{thm:Bairestratification-classical} cannot be avoided. For instance, there are arbitrarily complex Borel functions between the real line \( \RR \) and the Baire space \( \pre{\o}{\o}\), but the closure under pointwise limits of the class of continuous functions between such spaces reduces to the collection of constant functions.

Moving to generalized descriptive set theory, Borel and \( \boldsymbol{\Sigma}^0_\xi \)-measurable functions admit straightforward generalizations: \( \l^+ \)-Borel measurable and \( \Sii{0}{\xi} \)-measurable functions (Section~\ref{subsec:measurablefunctions}). 
As we will see in Section~\ref{sec:Borelfunctions}, the analysis of pointwise limits is instead more surprising. 
To simplify the present discussion, let us temporarily assume that \( \l \) is regular and work with \( \l \)-metrizable spaces (see Section~\ref{subsec:metr}). Although the topology on such spaces is completely determined by \( \l \)-limits, this kind of limits are no longer sufficient to generate the collection of \( \l^+ \)-Borel measurable functions.
In fact, 
when \( \l > \o \) the closure under \( \l \)-limits of the class of continuous functions is precisely the collection of all functions which are \( \Sii{0}{n} \)-measurable for some finite \( n \geq 1 \) 
(Corollary~\ref{cor:measurable_limits}). This forces us to consider other well-studied kinds of pointwise limits, i.e.\ limits over directed sets (Section~\ref{subsec:limits}), and eventually get the following result. 

\begin{theorem} \label{thm:closureunderlimits-generalized}
Let \( X \) and \( Y \) be \( \l \)-metrizable spaces of weight at most \( \l \). Then the class of \( \l^+ \)-Borel measurable functions between \( X \) and \( Y \) coincides with the closure under pointwise \( D \)-limits of the class of continuous functions, where \( D \) varies among all directed sets of size at most \( \l \).
\end{theorem}

The analogue of this theorem in the classical setting \( \l = \o \) holds as well, but it is subsumed by the stronger Theorem~\ref{thm:closureunderlimits-classical}, which drastically reduces the limits to be employed to a single sequential limit. 
Although this is no longer possible in the generalized setup \( \l > \o \), we observe that it is enough to use sequential limits together with limits over the partial order \( \Fin_\l \) of finite subsets of \( \l \) (Theorem~\ref{teo:D_limits}), or even just a single kind of non-sequential limit, namely, \( \widehat{\Fin}_\l \)-limits (Theorem~\ref{teo:D_limitsunique}). 
Notice also that in Theorem~\ref{thm:closureunderlimits-generalized} there is no additional hypothesis on the spaces involved: this difference from Theorem~\ref{thm:closureunderlimits-classical} is only apparent, though, as \( \lambda \)-metrizability implies zero-dimensionality when \( \lambda > \o \) (see Theorem~\ref{theorem:sikorski}).

Getting a higher analogue of Theorem~\ref{thm:Bairestratification-classical} is an even more delicate matter, 
addressed in Section~\ref{sec:long}.
It turns out that there are serious obstacles at limit levels of cofinality smaller than \( \l \). Nevertheless, we managed to find a reasonable definition of generalized Baire class \( \xi \) functions (Definition~\ref{baire_class}) and, using an argument quite different from the classical one, prove the following result (see Theorem~\ref{baire_class_xi_teo}).

\begin{theorem} \label{thm:Bairestratification-generalized}
Let \( X \) and \( Y \) be \( \l \)-metrizable spaces of weight at most \( \l \), and assume that \( Y \) is spherically complete. Let \( \xi < \l^+ \). Then \( f \colon X \to Y \) is a \( \l \)-Baire class \( \xi \) function if and only if it is \( \Sii{0}{\xi+1} \)-measurable. 
\end{theorem}

Although somewhat technical for the problematic levels \( \xi \), our definition of \( \l \)-Baire class \( \xi \) functions is close to optimal, as shown by the various counterexamples presented in Section~\ref{subsec:longregular}. The additional spherically completeness hypothesis on \( Y \) can be avoided using slight variations of such definition (see the discussion after Theorem~\ref{baire_class_xi_teo}).

The case of singular cardinals \( \l \) needs further adjustments, but one can still get results along the lines of Theorems~\ref{thm:closureunderlimits-generalized} and~\ref{thm:Bairestratification-generalized} (see Section~\ref{subsec:singular} and, in particular, Theorem~\ref{baire_class_xi_teo_sing}).

Along the way, we prove other structural results concerning \( \Sii{0}{\xi} \)-sets and \( \Sii{0}{\xi} \)-measurable functions, among which it is worth mentioning the following ones:
\begin{enumerate-(a)}
\item 
Higher analogues of structural properties such as the reduction property, the separation property, and alike (Section~\ref{sec:structuralproperties}).
\item 
A characterization of \( \l \)-Baire class \( 1 \) functions in terms of limits of surprisingly simple Lipschitz functions, called \( \l \)-full functions (Section~\ref{sec:baire_class_1}).
\item 
A characterization of \( \Sii{0}{\xi} \)-measurable functions in terms of \emph{uniform} limits of simpler functions (Section~\ref{sec:uniform_lim}).
\end{enumerate-(a)}

The last two items are the counterparts in the generalized context of some classical results coming from~\cite{MR09}.

\section{Preliminaries}

Throughout the paper we work in \( \ZFC \) and, unless otherwise specified, assume that \( \l \) is an uncountable cardinal satisfying \( 2^{< \l} = \l \). We also let $\mu = \cof (\lambda) $. (Although some of the results would work with any infinite regular cardinal \( \mu \).) Moreover, all topological spaces are tacitly assumed to be regular and Hausdorff, unless otherwise specified.
We assume some familiarity with set theory and general topology, and we adopt the standard notation in those fields. For all undefined notions, the reader is referred to~\cite{Jech,ER}.

\subsection{Trees and spaces of sequences} \label{subsec:treesandspacesofsequences}

Given ordinals \( \a \) and \( \b \), we denote by \(\pre{\a}{\b}\) the set of all sequences of order type \(\a\) and values in \(\b\), 
and for \( s \in \pre{\a}{\b} \) we let \( \leng(s) = \a \) be its length. 
We also set \(\pre{<\a}{\b}= \bigcup_{\g<\a}\pre{\g}{\b}\). 
We write \( s \restriction \g \) for the restriction of \( s  \) to \( \g \leq \leng(s) \), and denote by \( s {}^\smallfrown t \) the concatenation of \( s \) and \( t \); as usual, when \( t = \langle \gamma \rangle \) has length \( 1 \), we write \( s {}^\smallfrown \g \) instead of \( s {}^\smallfrown \langle \g \rangle \). 
The sequences $s$ and $t$ are comparable if \(s \subseteq t\) or \(t \subseteq s\), and incomparable otherwise.
We let \( i^{(\g)}\) be the constant sequence with length \(\g\) and value \( i \).

A set \(\tr \subseteq \pre{\a}{\b}\) is called \markdef{tree} if it is closed under initial segments, while its \markdef{body} is the set  
\( [\tr]=\{x \in \pre{\a}{\b} \mid \forall \g <\a \, (x \restriction \g \in \tr) \} \). 
Notice that, formally, \( [\tr] \) depends on the ordinal \( \a \); however, the latter will always be clear from the context, and thus it can safely be omitted from the notation. 
A tree $\tr$ is \markdef{pruned} if for all \(s \in \tr\) there is $x \in [\tr]$ such that $s \subseteq x$. 

Given two cardinals \( \kappa \geq 2 \) and \( \nu \geq \omega \), we equip the set \( \pre{\nu}{\kappa} \) with the \markdef{bounded topology} \( \tau_b \), i.e.\ the topology generated by the basic open sets of the form \( \Nbhd_s = \{ x \in \pre{\nu}{\kappa} \mid s \subseteq x \} \), for \( s \in \pre{<\nu}{\kappa} \). 
Equivalently, \( \tau_b \) is the unique topology on \( \pre{\nu}{\kappa}\) whose closed sets coincide with the collection of all sets of the form \( [ \tr ] \), for \( \tr \subseteq \pre{\nu}{\kappa} \) a pruned tree. 
This allows us to canonically associate to each set \( A \subseteq \pre{\nu}{\kappa} \) the pruned tree \( \tr_A = \{ x \restriction \alpha \mid x \in A, \alpha < \nu \} \), called the \markdef{tree of \( A \)}, which has the property that \( [\tr_A ] = \cl(A) \). The notion of closed set can be strengthened by imposing further conditions on the tree associated to it. More precisely, we say that a tree \( \tr \subseteq \pre{<\nu}{\kappa} \) is superclosed if it is pruned and \( < \nu \)-closed, that is, for all limit \( \alpha < \nu \) and \( s \in \pre{\alpha}{\kappa} \) we have \( s \in \tr \) whenever \( s \restriction \beta \in \tr \) for all \( \beta < \alpha \). Accordingly, a set \( C \subseteq \pre{\nu}{\kappa}\) is  \markdef{superclosed} if \( C = [\tr] \) for some superclosed tree \( \tr \subseteq \pre{\nu}{\kappa} \).

\subsection{Generalized metrics}\label{subsec:metr}

Metric spaces are central in classical descriptive set theory.
When moving to the generalized context, we can continue to use classical metrics if \( \mu = \omega \) (\cite{DMR}), while if \( \mu > \omega \) then \( \mathbb{R} \)-valued metrics needs to be replaced with \( \GG \)-metrics, that is, metrics taking value in a totally ordered (abelian) group \( \GG \) with degree \( \mu \) (\cite{AMRS23, phdago}).
Since it turns out that in the latter case the choice of \( \GG \) is irrelevant, we usually speak of \( \mu \)-metrizable spaces rather than 
\( \GG \)-metrizable spaces. For the same reason, we can safely assume that \( \GG \) is always a field; for the sake of definiteness, we indeed stipulate that \( \GG = \RR \) if \( \mu = \omega \), and \( \GG = \mu\text{-}\QQ \),
where \( \mu \)-\( \QQ \) is Asperò-Tsaprounis' ordered field of $\mu$–rationals (\cite{AST18}), if \( \mu > \omega \). We also fix once and for all a strictly decreasing sequence \( (r_\a)_{\a < \mu } \) coinitial in \( \GG^+ \): this is done by letting \( \hat r = 2 \) if \( \mu = \omega \) and \( \hat r = \omega \) if \( \mu > \omega \), and then setting \( r_\alpha = \hat r^{-\alpha}\). 

It is well-known that if \( \mu > \omega \), all \( \mu \)-metrizable spaces are \( \mu \)-additive, and they indeed admit a compatible \emph{\( \GG \)-ultrametric}. (This follows e.g.\ from Theorem~\ref{theorem:sikorski}.) Moreover, most of the metric related notions can be naturally adapted to generalized metrics: this includes Cauchy-completeness, which in this case refers to Cauchy sequences of length \( \mu \) rather than countable sequences. 
A stronger notion of completeness is obtained by requiring that the intersection of any decreasing sequence of closed balls is nonempty. A \( \GG \)-metric satisfying this property is called \markdef{spherically complete}. One can show that spherically complete \( \GG \)-metrics are always Cauchy-complete, but the converse might fail --- 
indeed there are even spaces which admit a compatible Cauchy-complete \( \GG \)-metric, but no spherically complete \( \GG \)-metric is compatible with their topology.

Of particular interest in generalized descriptive set theory are the \markdef{generalized Cantor space} \( \pre{\l}{2} \) and the \markdef{generalized Baire space} \( \pre{\mu}{\l} \), equipped with their bounded topology. 
The space \( \pre{\l}{2} \) is homeomorphic to a superclosed subset of \( \pre{\mu}{\l} \), and it is indeed homeomorphic to the whole \( \pre{\mu}{\l} \) if and only if \( \l \) is not weakly compact.
Both spaces are (regular Hausdorff) topological spaces of weight and density character \( \l \). Moreover, they admit natural spherically complete \( \GG \)-ultrametrics which are compatible with their topology. In the case of \( \pre{\mu}{\l} \), for distinct \( x,y \in \pre{\mu}{\l} \) we set \( d(x,y) = r_\alpha \), where \( \alpha < \mu \) is smallest such that \( x \restriction \alpha \neq y \restriction \alpha \). The case of \( \pre{\l}{2} \) is similar: we fix a strictly increasing sequence \( (\l_\a)_{\a < \mu} \) of ordinals cofinal in \( \l \), and then for distinct \( x,y \in \pre{\l}{2} \) we set \( d(x,y) = r_\alpha \) for the smallest \( \alpha < \mu \) such that \( x \restriction \l_\alpha \neq y \restriction \l_\alpha \). 
When considering \( \pre{\mu}{\l}\) and \( \pre{\l}{2}\) as \( \GG \)-metric spaces, we always tacitly refer to these specific \( \GG \)-metrics.

A very convenient result in the context of \( \mu \)-metrizable spaces is the following.

\begin{theorem}[{\cite{SIK, AMRS23, phdago}}]\label{theorem:sikorski}
Suppose that \( \mu > \omega \). For any space X of weight at most \( \l\), the following are equivalent:
\begin{enumerate-(1)}
\item \label{theorem:sikorski-1}
\(X\) is \(\mu\)-metrizable;
\item \label{theorem:sikorski-2}
\(X\) is \(\mu\)-ultrametrizable;
\item \label{theorem:sikorski-3}
\(X\) is homeomorphic to a subset of \(\pre{\mu}{\l}\).
\end{enumerate-(1)}
Moreover, \( X\) is a spherically complete \( \mu \)-(ultra)metrizable space if and only if it is homeomorphic to a superclosed subset of \( \pre{\mu}{\l} \).
\end{theorem}

Given the importance that the spaces from Theorem~\ref{theorem:sikorski} play in this paper,
we denote by \( \metr \) the collection of all \( \mu \)-metrizable spaces of weight at most \( \l \).

With a little care, one can strengthen the implication \ref{theorem:sikorski-2} \( \Rightarrow\) \ref{theorem:sikorski-3} of Theorem~\ref{theorem:sikorski} and get the following result, whose proof is a higher analogue of~\cite[Theorem 4.1]{MR09} and works also when \( \mu = \omega \). 

\begin{proposition} \label{prop:bilip}
Let \( (X,d) \) be a \( \GG \)-ultrametric space of weight at most \( \l \). Then there is a topological embedding \( h \colon X \to \pre{\mu}{\l} \) such that \( d(h(x),h(y)) \leq d(x,y)  \), for all \( x,y \in X \). 

Moreover, we can further ensure that if \( d \) is spherically complete, then the range of \( h \) is a superclosed subset of \( \pre{\mu}{\l} \).
\end{proposition}

\begin{proof}
We define the inverse \( f = h^{-1}  \) of \( h \) by building a scheme \( \{ B_s \mid s \in \pre{<\mu}{\l} \}\) on \( X \) such that for every \( \a < \mu \) and \( s \in {}^{<\mu}{\l} \):
\begin{enumerate-(i)}
\item \label{Lip-1}
\( B_s \subseteq B_t \) whenever \( t \subseteq s \);
\item \label{Lip-2}
\( \{ B_t \mid \leng(t) = \a \} \) is a covering of \( X \) such that \( B_t \cap B_{t'} = \emptyset \) for all distinct \( t,t' \in \pre{\a}{\l}\);
\item \label{Lip-3}
if \( \leng(s) \) is a successor ordinal and \( B_s \neq \emptyset \), then \( B_s \) is the open ball \( B_s = B_d(x,r_{\leng(s)}) \), for some/any \( x \in B_s \); if \( \leng(s) \) is a limit ordinal, then \( B_s = \bigcap_{\beta < \leng(s)} B_{s \restriction \beta+1}\). 
\end{enumerate-(i)}
Condition~\ref{Lip-3} ensures that \( \diam(B_s) \leq r_{\leng(s) }\) whenever \( \leng(s)  \) is a successor ordinal. Together with condition~\ref{Lip-1}, this implies that the map \( f \) canonically induced by the scheme, which is defined by letting \( f(x) \) be the unique element in \( \bigcap_{\a < \mu} B_{x \restriction \a} \) if the latter is nonempty (otherwise \( f(x) \) is undefined), is well-defined and continuous. 
Condition~\ref{Lip-2} entails that \( f \) is a bijection on its domain. Moreover, \( f \) is an open map because each \( B_s \) is clopen by condition~\ref{Lip-3} and \( \mu \)-additivity of \( X \), 
together with the fact that~\ref{Lip-2} grants that \( f(\Nbhd_s \cap \dom(f)) = B_s \).
Therefore \( h = f^{-1} \colon X \to \pre{\mu}{\l} \) is a topological embedding, and we want to check that \( d(x,y) \leq d(f(x),f(y)) \) for every \( x ,y \in \dom(f) \subseteq \pre{\mu}{\l} \). This is clear if \( x=y \), so suppose that \( x \neq y \). Let \( \a < \mu \) be smallest such that \( x \restriction \a \neq y \restriction \a \), so that \( d(x,y)=r_\a \). Necessarily, \( \a \) is a successor ordinal.
Since \( f(x) \) witnesses \( B_{x \restriction \a} \neq \emptyset \), the latter is an open ball of radius \( r_\a \) by condition~\ref{Lip-3}.
Suppose towards a contradiction that \( d(f(x),f(y)) < r_\alpha \). Then \( f(y) \in B_{x \restriction \a }\). On the other hand, \( f(y) \in B_{y \restriction \a} \) by definition of \( f \), hence \( B_{x \restriction \a} \cap B_{y \restriction \a} \neq \emptyset \). This contradicts~\ref{Lip-2} and the choice of \( \a \). Therefore \( d(x,y) = r_\a \leq d(f(x),f(y)) \). 
The fact that the image of \( h \) is superclosed when \( d \) is spherically complete easily follows from~\ref{Lip-2} and~\ref{Lip-3}.

It remains to construct the scheme \( \{ B_s \mid s \in \pre{<\mu}{\l} \}\), and this is done by recursion on \( \leng(s) \). 
Set \( B_\emptyset = X \), and \( B_s = \bigcap_{\beta < \leng(s)} B_{s \restriction \beta+1} \) if \( \leng(s) \) is limit. 
(Notice that in the latter case condition~\ref{Lip-2} is automatically satisfied since by inductive hypothesis it already holds at all levels \( \alpha < \leng(s) \).) 
Consider now a successor ordinal \( \alpha = \beta+1 \), and assume that by inductive hypothesis \( B_s \) has been defined for all \( s \in \pre{\b}{\l} \). Fix any \( s \in \pre{\beta}{\l} \).
If \( B_{s} = \emptyset \), then we set \( B_{s {}^\smallfrown i} = \emptyset \) for all \( i < \l \). If instead \( B_{s } \neq \emptyset \), then \( B_d(x,r_{\a}) \subseteq B_{s } \) for every \( x \in B_{s } \); this is because by condition~\ref{Lip-3} we have that \( B_{s} = B_d(x,r_\beta) \) if \( \beta \) is a successor ordinal, or \( B_{s } = \bigcap_{\gamma < \beta} B_d(x,r_{\gamma+1}) \) if \( \beta \) is limit. Fix an enumeration \( (x_i)_{i < \l } \), possibly with repetitions, of a dense subset \( D \) of \( B_{s} \). 
We define \( B_{s {}^\smallfrown i} \) by recursion on \( i < \l \), ensuring along the construction that \( B_{s {}^\smallfrown i} \) is an open ball with radius \( r_\a \). First we set \( B_{s {}^\smallfrown 0} = B_d(x_0,r_\a) \). For \( i > 0 \), we let \( A =  B_s \setminus \bigcup_{j < i} B_{s {}^\smallfrown j} \) and distinguish two cases. If \( A = \emptyset \), then we let \( B_{s \conc i} = \emptyset \). If instead \( A \neq \emptyset\), then \( A = \bigcup_{x \in A } B_d(x,r_\a) \) is a nonempty open set, and we can set \( B_{s \conc i} = B_d(x_k,r_\a) \) for the smallest \( k < \l \) such that \( x_k \in A \). Conditions~\ref{Lip-1} and~\ref{Lip-3} are satisfied by construction. Condition~\ref{Lip-2}, instead, follows from the fact that it holds at level \( \beta \) by inductive hypothesis, together with the fact that \( \{ B_{s \conc i} \mid i < \l \} \) is a covering of \( B_s \) (by density of \( D \)) and by construction \( B_{s \conc i} \cap B_{s \conc j} = \emptyset \) if \( i \neq j \).
\end{proof}

It might be worth recording that
Proposition~\ref{prop:bilip} entails an analogue of Theorem~\ref{theorem:sikorski} for the case \( \mu = \omega \). 

\begin{theorem}\label{theorem:sikorski2}
Suppose that \( \cof(\l) = \omega \). For any space X of weight at most \( \l\), the following are equivalent:
\begin{enumerate-(1)}
\item \label{theorem:sikorski2-1}
\(X\) is metrizable and \( \dim(X) = 0 \), i.e., \( X \) has Lebesgue covering dimension \( 0 \);
\item \label{theorem:sikorski2-2}
\(X\) is ultrametrizable;
\item \label{theorem:sikorski2-3}
\(X\) is homeomorphic to a subset of \(\pre{\omega}{\l}\).
\end{enumerate-(1)}
Moreover, \( X\) is a completely metrizable space with \( \dim(X) = 0 \) if and only if it is completely ultrametrizable, if and only if it is homeomorphic to a closed subset of \( \pre{\omega}{\l} \).
\end{theorem}

\begin{proof}
The equivalence between~\ref{theorem:sikorski2-1} and~\ref{theorem:sikorski2-3} is standard (see e.g.\ \cite[Proposition 3.3.2]{DMR}). The implication~\ref{theorem:sikorski2-3} \( \Rightarrow \) \ref{theorem:sikorski2-2} is obvious, while the reverse implication~\ref{theorem:sikorski2-2} \( \Rightarrow \) \ref{theorem:sikorski2-3} follows from Proposition~\ref{prop:bilip}.
\end{proof}

\subsection{Limits} \label{subsec:limits}

Let \( X \) be a topological space, and let \( (D,\leq) \) be a directed set, still denoted by \( D \). A point \( x \in X  \) is a \markdef{\( D \)-limit} of a \( D \)-net \( (x_d)_{d \in D} \) of points from \( X \) if for every open neighborhood \( U \) of \( x \) there is \( d \in D \) such that \( x_{d'} \in U \) for all \( d' \in D \) with \( d' \geq d\). When this happens, we write \( x = \lim_{d \in D} x_d\). An important case often considered in (generalized) metrizable spaces is when \( D = \alpha \) is a limit ordinal, i.e.\ \markdef{sequential limits} (or \markdef{\( \alpha \)-limits}).

It is well-known that \( D \)-limits capture the notion of topological closure: a set \( C \subseteq X \) is (topologically) closed in an arbitrary topological space \( X \) if and only if \( \lim_{d \in D} x_d \in C \) for all directed sets \( D \) and all nets \( (x_d)_{d \in D} \) of points from \( C \). 
If \( X \) has weight \( \lambda \), then without loss of generality one can restrict the attention to directed sets \( D \) of size at most \( \lambda \), while if \( X \) is \( \mu \)-metrizable, then it is enough to consider \( \mu \)-limits. Occasionally we will also consider restricted forms of \( D \)-limits, i.e.\ we will restrict to \( D \)-limits satisfying some special extra properties --- see Section~\ref{sec:long}.

Limits can be composed in the obvious way: if \( D,D' \) are two directed sets and and \( (x_{d,d'})_{d \in D, d' \in D'} \) is a family of points of \( X \), we let \( \lim_{d \in D} \lim_{d' \in D'} x_{d,d'} \) be the point (if it exists) \( x = \lim_{d \in D} x_d \), where in turn \( x_d = \lim_{d' \in D'} x_{d,d'} \) for all \( d \in D \). The previous double limit operator will often be denoted by \( \lim_D \circ \lim_{D'} \). 

The notion of \( D \)-limit can be lifted to functions by taking pointwise limits: if \( f \)  and \( (f_d)_{d \in D} \) are functions between topological spaces \( X \) and \( Y \), then we write \( f = \lim_{d \in D} f_d \) if \( f(x) = \lim_{d \in D} f_d(x) \) for all \( x \in X \). All the previous concepts and notations can be adapted to pointwise limits of functions in the obvious way. 

If \( \mathcal{F} \) is a collection of function from \( X \) to \( Y \), we let 
\[ 
\Dlim \F  = \left\{ \lim_{d \in D} f_d \mid f_d \in \mathcal{F} \right\}
\] 
be the collection of all pointwise \( D \)-limits of functions from \( \mathcal{F} \). 
This can be extended to arbitrary families of limits by setting, for every collection \( \mathcal{D} \) of directed sets,
\( \DDlim \F = \bigcup_{D \in \mathcal{D}} \Dlim \mathcal{F} \).

\subsection{Boldface pointclasses and Borel structure} \label{subsec:measurablefunctions}

A \markdef{pointclass} \(\boldsymbol{\Gamma}\) is an operation assigning to every nonempty topological space $X$ a nonempty family \(\boldsymbol{\Gamma}(X)\subseteq \pow(X)\). A pointclass \(\boldsymbol\Gamma\) is
said to be \markdef{boldface} if it is closed under continuous preimages, that is, if $f \colon X \to Y$ is continuous and $B \in \boldsymbol\Gamma(Y)$, then $f^{-1}(B) \in \boldsymbol\Gamma(X)$. 
The dual $\check{\boldsymbol\Gamma}$ of \( \boldsymbol{\Gamma}\) is the boldface pointclass defined by $\check{\boldsymbol\Gamma}(X)= \{X \setminus A \mid A \in \boldsymbol\Gamma(X)\}$, while the ambiguous pointclass associated to $\boldsymbol\Gamma$ is given by $\boldsymbol\Delta_{\boldsymbol\Gamma}(X) = \boldsymbol\Gamma(X) \cap \check{\boldsymbol\Gamma}(X)$.
Examples of boldface pointclasses are the classes \( \Bor \), \( \Sii{0}{\xi} \), \( \Pii{0}{\xi} \), and \( \Dee{0}{\xi} \) defined below.
Let \(\boldsymbol{\Gamma} \) be a boldface pointclass, \( X \) be a topological space, and \(A \in \boldsymbol{\Gamma}(X)\).
A \markdef{\( \boldsymbol{\Gamma}\)-covering} of \( A \) is a family \( \{ A_i \mid i \in I \} \), for some index set \( I \), such that \( \bigcup_{i \in I } A_i =  A \) and \( A_i \in \boldsymbol{\Gamma}(X) \) for all \( i \in I \). A \( \boldsymbol{\Gamma} \)-covering \( \{ A_i \mid i \in I \} \) is said to be \markdef{disjoint} if \( A_i \cap A_j = \emptyset \) for all distinct \( i,j \in I \). (But notice that some \( A_i \) might be empty.)
Finally, a \markdef{\(\boldsymbol{\Gamma}\)-partition} of \(A \) is a disjoint \( \boldsymbol{\Gamma} \)-covering of \( A \) all of whose elements are nonempty. When the class \( \boldsymbol{\Gamma} \) is irrelevant, we simply drop it from the above terminology.

Let \( X \) be a topological space. A set \( A \subseteq X \) is \markdef{\( \l^+  \)-Borel} if it belongs to the \( \l^+  \)-algebra generated by the topology of \( X \), which is denoted by \( \Bor \). 
As in the classical case, the collection of $\l^+$-Borel subsets of a topological space \((X,\tau)\) can be stratified in a hierarchy by setting \(\Sii{0}{1}(X) =\tau\) and, for \(\xi> 1\),
\[
\Sii{0}{\xi}(X) =   \left\{ \bigcup\nolimits_{i<\l}A_i  \mid X \setminus A_i  \in  \bigcup\nolimits_{ \xi' < \xi}  \Sii{0}{\xi'}(X) \text{ for all } i<\l \right\}.
\]
The resulting pointclass \( \Sii{0}{\xi} \) is boldface. We denote its dual by \( \Pii{0}{\xi} \) and the associated ambiguous pointclass by \( \Dee{0}{\xi} \).
It is not hard to prove that  
\( \Bor(X) = \bigcup_{1 \leq \xi < \l^+} \Sii{0}{\xi}(X) \).
Moreover, \( \Dee{0}{\xi}(X) \subseteq \Sii{0}{\xi}(X) \subseteq \Dee{0}{\xi'}(X)\) for every \(\xi < \xi'\), and the same is true with \( \Sii{0}{\xi} \) replaced by \( \Pii{0}{\xi} \). 
Such inclusions are strict whenever \( X \) contains a copy of \( \pre{\l}{2} \), and in this case we say that the \( \l^+ \)-Borel hierarchy does not collapse. 
This essentially follows from the fact that there are \( \pre{\l}{2} \)-universal sets for both \( \Sii{0}{\xi}(\pre{\l}{2}) \) and \( \Pii{0}{\xi}(\pre{\l}{2}) \).
The class \( \Sii{0}{\xi}(X) \) is trivially closed under unions of size \( \l \). If \( \xi > 1 \) is a successor ordinal, then it is also closed under intersections of size \( < \mu \); if instead \( \xi \) is a limit ordinal, it is closed under intersections of size \( < \cof(\xi) \). Under mild assumptions on \( X \), this is optimal for spaces of weight at most \( \l \). Dual properties hold for \( \Pii{0}{\xi}(X) \), and therefore \( \Dee{0}{\xi}(X) \) is a \( \mu \)-algebra if \( \xi > 1 \) is a successor ordinal, or a \( \cof(\xi) \)-algebra if \( \xi \) is a limit ordinal. 
When \( \xi = 1 \), closure properties depend instead on the additivity of the space.
For example, if \( X \) is \( \mu \)-metrizable, then \( \Dee{0}{1}(X) \) forms again a \( \mu \)-algebra.
We refer the reader to~\cite{ACMRP} for more information on the \( \l^+ \)-Borel hierarchy.
When the space \( X \) is clear from the context, we drop it from all the notation above; on the other hand, when needed we might add a reference to the topology \( \tau \) of \( X \) and write e.g.\ \( \Sii{0}{\xi}(\tau) \).

Given two topological spaces \( X \) and \( Y \), a function \(f \colon X \to Y \) is \markdef{\(\l^+ \)-Borel measurable} if \( f^{-1}(U) \in \Bor(X) \) for all \( U \in \Sii{0}{1}(Y) \) (equivalently, for all \( U \in \Bor(Y) \)). 
If \( Y \) has weight at most \( \lambda \), then for every \( \l^+ \)-Borel function \( f \colon X \to Y \) there is \( 1 \leq \xi < \lambda^+ \) such that \( f^{-1}(U) \in \Sii{0}{\xi}(X) \) for all open sets \( U \subseteq Y \): in this case, we say that \( f \) is \markdef{\( \Sii{0}{\xi} \)-measurable}. The collection of all \( \Sii{0}{\xi} \)-measurable functions from \( X \) to \( Y \) is denoted by \( \M_\xi(X,Y) \), and to simplify the notation we also write \( \M_{<\xi}(X,Y)\) instead of \(\bigcup_{\xi' < \xi} \M_{\xi'}(X,Y) \). In a similar fashion, one can also define \( \Dee{0}{\xi} \)-measurable functions and alike.

\section{Structural properties} \label{sec:structuralproperties}

Following~\cite{DMR}, we consider higher analogues of the structural properties from~\cite[Section 22.C]{Kech}.
Given sets \(X \),\( Y\), and \(P \subseteq X \times Y\), a \markdef{uniformization} of P is a subset \(P^* \subseteq P\) such that for all $x \in X$ 
\[
\exists y \, P(x,y) \iff \exists ! y \, P^*(x,y). 
\]

\begin{definition} 
Let $\boldsymbol\Gamma$ be a boldface pointclass. 
\begin{enumerate-(1)}
\item 
$\boldsymbol\Gamma$ has the \markdef{separation property} if for every $X \in \metr$ and all disjoint sets $A, B \in$ $\boldsymbol\Gamma(X)$, there is $C \in \boldsymbol\Delta_{\boldsymbol\Gamma}(X)$ such that $A \subseteq C$ and $C \cap B=\emptyset$. 
\item 
$\boldsymbol\Gamma$ has the \markdef{$\l$-separation property} if for every $X \in \metr$ and every sequence of sets \( (A_{i})_{i \in \l} \) from \( \boldsymbol\Gamma(X) \) satisfying $\bigcap_{i<\l} A_{i}=\emptyset$, there is a sequence \( (B_i)_{i < \l }\) of sets from $\boldsymbol\Delta_{\boldsymbol\Gamma}(X)$ such that $\bigcap_{i<\l} B_{i}=\emptyset$ and  $A_{i} \subseteq B_{i}$ for every \( i < \l \).
\item 
$\boldsymbol\Gamma$ has the \markdef{reduction property} if for every  $X \in \metr$ and every $A, B \in \boldsymbol{\Gamma}(X)$, there are disjoint sets $A^{*}, B^{*} \in \boldsymbol{\Gamma}(X)$ such that \( A^{*} \subseteq A \), \( B^{*} \subseteq B \), and $A^{*} \cup B^{*}=A \cup B$. 
\item 
$\boldsymbol\Gamma$ has the \markdef{$\l$-reduction property} if for every $X \in \metr$ and  every sequence  \( (A_{i})_{i < \l} \) of sets from \( \boldsymbol\Gamma(X) \), there is a sequence \( (A_{i}^{*})_{i < \l } \) of pairwise disjoint sets from \( \boldsymbol\Gamma(X) \) such that $\bigcup_{i<\l} A_{i}^{*}=\bigcup_{i<\l} A_{i}$ and $A_{i}^{*} \subseteq A_{i}$ for every \( i < \l \).
\item 
$\boldsymbol\Gamma$ has the \markdef{ordinal $\l$-uniformization property} if for every $X \in \metr$ and every $R \in \boldsymbol{\Gamma}(X \times \l)$, 
there is a uniformization of $R$ in $\boldsymbol{\Gamma}(X \times \l)$.
\end{enumerate-(1)}
\end{definition}

Clearly, if \(\boldsymbol\Gamma \) has the $\l$-separation property then it has the separation property, and if \(\boldsymbol\Gamma \) has the $\l$-reduction property then it has the reduction property. To state the relationships among the other structural properties we need one more definition.

\begin{definition} 
A boldface pointclass $\boldsymbol\Gamma$ is \markdef{$\l$-reasonable} if 
for every \( X \in \metr \), every set \( I \) with \( |I| \leq \l \), and every family \((A_i)_{i \in I}\) of subsets of $X$, we have that
$\forall i \in I \, (A_i \in \boldsymbol\Gamma(X))$ if and only if \( A \in \boldsymbol\Gamma(X \times I) \), where
\[
A=\{(x,i) \in X \times I \mid x \in A_i\}.
\]
\end{definition}

If a boldface pointclass \(\boldsymbol\Gamma \supseteq \Dee{0}{1}\) is such that \( \boldsymbol{\Gamma}(X) \) is closed under unions of size $\l$ and intersections with clopen sets (for every \( X \in \metr \)), then  \(\boldsymbol\Gamma\) is $\l$-reasonable, which implies that also $\check{\boldsymbol\Gamma}$ and 
$\boldsymbol\Delta_{\boldsymbol\Gamma}$ are \( \l \)-reasonable. 
Therefore, all of \(\Sii{0}{\xi} \), \(\Pii{0}{\xi} \), and \( \Dee{0}{\xi} \) are $\l$-reasonable.

The following is the higher analogue of~\cite[Proposition 22.15]{Kech}, and can be proved using similar arguments.

\begin{proposition} \label{prop:regularityproperties}
Let $\boldsymbol\Gamma$ be a boldface pointclass. 
\begin{enumerate-(1)}
\item  \label{prop:regularityproperties-1}
If $\boldsymbol\Gamma$ has the reduction property, then $\check{\boldsymbol\Gamma}$ has the separation property.
\item  \label{prop:regularityproperties-2}
If $\boldsymbol\Gamma$ is closed under unions of length $\l$ and has the $\l$-reduction property, then $\check{\boldsymbol\Gamma}$ has the $\l$-separation property.
\item \label{prop:regularityproperties-3}
If $\boldsymbol\Gamma$ is $\l$-reasonable, then $\boldsymbol\Gamma$ has the $\l$-reduction property if and only if
$\boldsymbol\Gamma$ has the ordinal $\l$-uniformization property.
\item  \label{prop:regularityproperties-4}
If there is a $\pre{\l}{2}$-universal
set for $\boldsymbol\Gamma( \pre{\l}{2})$, then $\boldsymbol\Gamma$ cannot have both the reduction and the separation properties.
\end{enumerate-(1)}
\end{proposition} 

The separation property admits a technical variation that will be used later on.

\begin{corollary}(Folklore)\label{separation_partition}
Suppose that $\boldsymbol\Gamma$ has the separation property. Let \( X \in \metr \) be such that \( \boldsymbol{\Gamma}(X) \) is closed under unions and intersections of size at most \( \nu \), for some cardinal \( \nu \). 
Let \( C \in \boldsymbol{\Delta}_{\boldsymbol{\Gamma}}(X) \) and \( (P_i)_{i < \nu} \) be a family of pairwise disjoint nonempty \( \boldsymbol{\Gamma}(X)\)-subsets of \( C \). Then there is a \( \boldsymbol{\Delta}_{\boldsymbol{\Gamma}} \)-partition \( \{ C_i \mid i < \nu \} \) of \( C \) such that \( P_i \subseteq C_i \) for every \( i < \nu \).
\end{corollary}
 
\begin{proof}
For each \( i < \nu \), let \( D_i = \bigcup_{j \neq i} P_j\). Then \( P_i \) and \( D_i \) are disjoint \( \boldsymbol{\Gamma}\)-sets. Applying the separation property we get \( C'_i \in \boldsymbol{\Delta}_{\boldsymbol{\Gamma}}(X) \) such that \( P_i \subseteq C'_i \) and \( C'_i \cap P_j = \emptyset\) for every \( j < \nu \) different from \( i \). Since our hypotheses imply that \( \boldsymbol{\Delta}_{\boldsymbol{\Gamma}}(X)\) is a \( \nu^+ \)-algebra, it is enough to let \( C_0 = (C \cap C'_0) \cup \left(C \setminus \bigcup_{ j < \nu} C'_j \right) \) and \( C_i = (C \cap C'_i) \setminus \bigcup_{j < i} C'_i \), for \( 0 < i < \nu \).
\end{proof}

We now want to prove the analogue in generalized descriptive set theory of \cite[Theorem 22.16]{Kech}. The straightforward generalization of the original argument would require $\Sii{0}{\xi}(X)$ to be closed under intersections of length $\l$, which is true if \( \l \) is regular and $\xi$ is either a successor ordinal or a limit ordinal with cofinality \( \l \), but fails in all other cases.
As a remedy, following~\cite[Proposition 4.2.1]{DMR}, which proves the same result for the case \( \mu = \omega \), we can exploit the following observation.

\begin{lemma}\label{lem:reg}
Assume that \( \mu > \omega \). Let $X \in \metr$ and $1 \leq \xi<\l^+ $. 
\begin{enumerate-(1)}
\item \label{lem:reg1} 
If \(\xi\) is a successor ordinal and \(A\in \Sii{0}{\xi}(X)\), then \(A= \bigcup_{i<\mu} A_{i}\) for some family \( (A_i)_{i < \mu} \) of sets in \( \Dee{0}{\xi}(X) \). Moreover, we can assume that \( A_i \subseteq A_j \) for all \( i \leq j < \mu \) or, alternatively, that the sets \( A_i \) are pairwise disjoint. 
\item \label{lem:reg2} 
If \(\xi\) is a limit ordinal, \( ( \xi_i)_{i < \cof(\xi)} \) is a strictly increasing sequence cofinal in \( \xi \), and \(A\in \Sii{0}{\xi}(X)\), then \(A= \bigcup_{i<\cof(\xi)} A_{i}\) for some family \( (A_i)_{i < \cof(\xi)} \) such that \( A_i \in \Dee{0}{\xi_i+2} \) for all \( i < \cof(\xi) \).
Moreover, we can assume that \( A_i \subseteq A_j \) for all \( i \leq j < \cof(\xi) \) or, alternatively, that the sets \( A_i \) are pairwise disjoint.
\item \label{lem:reg3} 
If \(\nu<\l\) and \(A_j \in \Sii{0}{\xi}(X)\) for every \(j<\nu\), then \(\bigcap_{j<\nu} A_{j} \in \Dee{0}{\xi+1}(X)\). 
\end{enumerate-(1)}
The dual results obtained by replacing \( \Sii{0}{\xi} \) with \( \Pii{0}{\xi} \) and swapping the role of unions and intersections hold as well.
\end{lemma}

\begin{proof}
The dual results can be obtained by taking complements and using De Morgan's rules, so we only consider the case of the classes \( \Sii{0}{\xi}(X) \).
By Theorem \ref{theorem:sikorski}, we can assume that $X \subseteq \pre{\mu}{\l}$. 

We argue by induction on \( 1 \leq \xi < \l^+ \). For the basic case \(\xi=1\), observe that every \( A \in \Sii{0}{1}(X) \) can be written as \( \bigcup_{i < \mu} A_i \), where 
\[
A_i = \bigcup \{ \Nbhd_s \cap X \mid \leng(s) = i,  \Nbhd_s \cap X \subseteq A \} .
\]
Each \( A_i \) belongs to \( \Dee{0}{1}(X) \) because 
\[ 
X \setminus A_i = \bigcup \{ \Nbhd_t \cap X \mid \leng(t) = i , \Nbhd_t \cap X \not\subseteq A \} .
\]
Clearly, \( A_i \subseteq A_j \) whenever \( i \leq j \). If instead we want the sets \( A_i \) to be pairwise disjoint, we replace each \( A_i \) with \( A_i \setminus \bigcup_{j < i} A_i \): such sets are still clopen because \( \Dee{0}{1}(X) \) is a \( \mu \)-algebra (as \( X \in \metr \) is \( \mu \)-additive). This proves~\ref{lem:reg1}.
Part~\ref{lem:reg2} needs not to be considered in the case \( \xi = 1 \).
As for~\ref{lem:reg3}, it suffices to prove that \( \bigcap_{j < \nu} A_j \in \Sii{0}{2}(X) \). 
Using~\ref{lem:reg1}, write every \( A_j \) as \( A_j = \bigcup_{i< \mu} A_{j,i} \) with \( A_{j,i} \in \Dee{0}{1}(X) \). Then
\[
\bigcap_{j < \nu} A_j = \bigcap_{j < \nu} \bigcup_{i < \mu } A_{j,i} = \bigcup_{s \in \pre{\nu}{\mu}} \bigcap_{j < \nu } A_{j,s(j)}.
\]
If \( \l \) is regular, then our assumption \( 2^{< \l} = \l \) entails \( \l^{< \l} = \l \), and thus \( |\pre{\nu}{\mu} | \leq \l^{<\l} = \l \) because \( \mu = \l \). If \( \l \) is singular, then \( 2^{< \l} \) entails that \( \l \) is strong limit, and thus \( |\pre{\nu}{\mu}| \leq \kappa^\kappa = 2^\kappa < \l \) for \( \kappa = \max \{ \mu,\nu \} < \l \). Therefore in all cases \( |\pre{\nu}{\mu}| \leq \lambda \), and thus \( \bigcap_{j < \nu} A_j \in \Sii{0}{2}(X) \), as desired.

Let now $\xi > 1 \). 
Part~\ref{lem:reg1} is relevant only when \( \xi \) is successor, so assume that \( \xi = \xi'+ 1\). Write \( A = \bigcup_{j < \l} B_j\) with \( B_j \in \Pii{0}{\xi'}(X) \). Fix a strictly increasing sequence \( (\l_i)_{i < \mu} \) cofinal in \( \l \). Then each \( A_i = \bigcup_{j < \l_i } B_j  \) belongs to \( \Dee{0}{\xi'+1}(X) = \Dee{0}{\xi}(X) \) because~\ref{lem:reg3} holds at level \( \xi' \) by inductive hypothesis, and clearly \( A = \bigcup_{i < \mu} A_i \). Moreover, if \( i \leq j < \mu \) then \( A_i \subseteq A_j \) by construction. If instead we want the sets \( A_i \) to be pairwise disjoint, then we again replace each \( A_i \) with \( A_i \setminus \bigcup_{j < i} A_j \): since \( \Dee{0}{\xi}(X) \) is a \( \mu \)-algebra (because \( \xi \) is a successor ordinal), this works. 

If \( \xi \) is limit, we instead need to prove part~\ref{lem:reg2}. 
By definition, \(A= \bigcup_{j<\l} B_{j}\), where  \(B_{j} \in  \bigcup_{ \xi' < \xi} \Pii{0}{\xi'}(X) \). 
For all \(i<\cof(\xi)\), set \(A_{i} = \bigcup \{ B_{j} \mid   B_{j} \in  \Pii{0}{\xi_i}(X) \}\). 
Then \( A_i \in \Sii{0}{\xi_i+1}(X) \subseteq \Dee{0}{\xi_i+2}(X) \). Moreover, \( A_i \subseteq A_j \) if \( i \leq j \), and \( A = \bigcup_{i < \cof(\xi)} A_i \). If instead we want the sets \( A_i \) to be pairwise disjoint, we once again replace each \( A_i \) with \( A_i \setminus \bigcup_{j < i} A_j \), which belongs to \( \Dee{0}{\xi_i+2}(X) \) because the sets \( A_i \) are in \( \Sii{0}{\xi_i+1}(X) \).  

Finally, we prove \ref{lem:reg3}. Again, we only need to check that $\bigcap_{j<\nu} A_j \in \Sii{0}{\xi+1}$. 
For all $i<\nu$, apply%
\footnote{This can be done because in the previous two paragraphs we already proved that \ref{lem:reg1} and \ref{lem:reg2} hold at level $\xi$.}
\ref{lem:reg1} or~\ref{lem:reg2}, depending on whether $\xi$ is a successor or a limit ordinal, to get $A_j=\bigcup_{i<\k} B_{j,i}$ with \(B_{j,i} \in \Dee{0}{\xi}(X)\) and $\k=\mu$ if $\xi$ is successor or $\k=\cof(\xi)$ if \( \xi \) is limit.  
Arguing as in the case \( \xi = 1 \), we have
\( \bigcap_{j<\nu} A_j=\bigcup_{s \in \pre{\nu}{\k}} \bigcap_{j<\nu} B_{j, s(j)} \),
and we have to check that in all cases \( |\pre{\nu}{\k}| \leq \l \).
If \( \l \) is regular, then \( |\pre{\nu}{\k}| \leq \l^{< \l} = \l \) because \( \k \leq \l \).
If instead \( \l \) is singular, then \( \k < \l \): therefore \( |\pre{\nu}{\k}| \leq 2^{\max\{ \nu,\k \}} < \l \) because \( \l \) is strong limit.
\end{proof}

\begin{theorem}\label{22.16}
For any $1 < \xi<\l^+ $, the boldface pointclass $\Sii{0}{\xi}$ has the ordinal $\l$-uniformization property, and thus the $\l$-reduction
property, but it does not have the $\l$-separation property. The class $\Pii{0}{\xi}$  has the $\l$-separation property, but not the $\l$-reduction property.

The same is true for \( \xi = 1 \) if either \( \mu > \omega \), or \( \mu = \o \) and we restrict the attention to spaces \( X \in \metr \) with \( \dim(X) = 0 \) (equivalently: to ultrametrizable spaces of weight at most \( \l \)).
\end{theorem}

\begin{proof}
The case \( \mu = \omega \) has already been treated%
\footnote{Formally,~\cite[Proposition 4.4.4]{DMR} states the result just for \( \l \)-Polish spaces, i.e.\ \emph{completely} metrizable spaces of weight at most \( \l \). However, the proof goes trough also for the more general class \( \metr \).}
in~\cite[Proposition 4.4.4]{DMR}. We show that the same argument can be adapted to deal with the remaining cases, so from now on assume \( \mu > \omega \). 
By Proposition~\ref{prop:regularityproperties}, it is enough to show that $\Sii{0}{\xi}$ has the ordinal $\l$-uniformization property. Fix any \( X \in \metr \): since we assumed \( \mu > \omega \), by Theorem \ref{theorem:sikorski} we can suppose that \(X \subset \pre{\mu}{\l}\). We distinguish three cases.

First suppose that $\xi=1$. 
For any \(i<\mu\), we say that a set $A \subseteq X$ is $i$-clopen if there is $S \subseteq\pre{i}{\l}$ such that $A=\bigcup\left\{\Nbhd_s \cap X \mid s \in S\right\}$. It is easy to check that $i$-clopen sets are closed under complements, and arbitrary unions and intersections. Moreover, if $j \leq i$ then every $j$-clopen set is also $i$-clopen. 
Fix any $R  \in \Sii{0}{1}(X \times \l)$.
For every \(i<\mu\) and \(\g<\l\), let
\[
R_i^\gamma=\bigcup\left\{\Nbhd_s \cap X \mid s \in\pre{i}{\l} \wedge  \Nbhd_s \times\{\gamma\} \subseteq R\right\},
\]
and notice that \( \bigcup_{i < \mu} \bigcup_{\g < \l}  R_i^\g \) coincides with the projection of \( R \) on the first coordinate.
Each \(R_i^\gamma\) is  $i$-clopen, hence so is $R_i=\bigcup_{\g <\l } R^\g_i$. The set
\[
Q_i^\gamma = R^\g_i \setminus \left(\bigcup\nolimits_{i' < i} R_{i'} \cup \bigcup\nolimits_{\g' < \g} R^{\g'}_i \right)
\]
is $i$-clopen too, hence 
\[
\begin{aligned}
Q^* & =\left\{(x, \gamma, i) \in X \times \lambda \times \mu \mid x \in Q_i^\gamma\right\} \\
& =\bigcup_{i < \mu} \bigcup_{\gamma<\lambda} \left(Q_i^\gamma \times\{\gamma\} \times\{i\} \right) \in  \Sii{0}{1}(X \times \l \times \mu ).
\end{aligned}
\]
Finally, the set 
\[
R^* =\{ (x, \gamma) \in  X \times \lambda \mid \exists i < \mu \, [(x, \gamma, i) \in Q^*] \} 
\]
is open and uniformizes $R$.

Next we consider the case where $\xi=\xi'+1>1$ is a successor ordinal. 
Consider $R  \in \Sii{0}{\xi}(X \times \l)$, and write it as $R=\bigcup_{i <\l } R_i$ with $ R_i \in \Pii{0}{\xi'}(X \times \l)$. 
Let 
\[
Q=\left\{(x, \gamma, i) \in X \times \l \times \l \mid(x, \gamma) \in R_i\right\},
\]
and notice that $Q \in \Pii{0}{\xi'}(X \times \l \times \l )$ because $\Pii{0}{\xi'}$ is $\l$-reasonable. Endow $\l \times \l$ with the G\"odel well-ordering $\preceq$, and define 
\begin{align*}
 Q^* =\{ (x, \gamma, i) \in X \times \l \times \l \mid  (x, \gamma, i) \in Q \wedge \forall (\g',i') \prec (\g,i)  \left[\left(x, \gamma', i'\right) \notin Q\right]\}. 
\end{align*}
There are less than $\l$-many pairs $(\g',i') \prec (\g,i)$, and since the intersection of less than \( \l \)-many sets in \( \Sii{0}{\xi'}(X \times \l \times \l) \) is in \( \Dee{0}{\xi'+1}(X \times \l \times \l) = \Dee{0}{\xi}(X \times \l \times \l )\) by Lemma~\ref{lem:reg}\ref{lem:reg3}, we get that $Q^* \in  \Dee{0}{\xi}(X \times \l \times \l) \subseteq \Sii{0}{\xi}(X \times \l \times \l)$. 
Therefore the set 
\[
R^* =\{ (x, \gamma) \in  X \times \lambda \mid \exists i < \mu \, [(x, \gamma, i) \in Q^*] \} .
\]
belongs to $\Sii{0}{\xi}(X \times \l)$ and  uniformizes $R$. 

Finally, assume that $\xi$ is a limit ordinal, and let $R \in \Sii{0}{\xi}(X \times \l)$. 
By Lemma~\ref{lem:reg}\ref{lem:reg2}, there are an increasing sequence \( (\xi_i)_{i < \cof(\xi)} \) of ordinals smaller than \( \xi \) and sets \( R_i \in \Dee{0}{\xi_i}(X \times \l) \) such that $R=\bigcup_{i<\cof(\xi)} R_i$. 
For $i<\cof(\xi)$, let $Q^i \subseteq X \times \l \times \cof(\xi)$ be defined by
\[
Q^i=\left\{\left(x, \gamma, i'\right) \in X \times \l \times \cof(\xi) \mid i' \leq i  \wedge (x, \gamma) \in R_{i'}\right\},
\]
and set
\( Q=\bigcup_{i<\cof(\xi)} Q^i=\left\{(x, \gamma, i) \in X \times \l \times \l \mid(x, \gamma) \in R_i\right\} \).
Since each pointclass \(\Sii{0}{\xi_i} \) is \( \l \)-reasonable, then \( Q^i \in \Sii{0}{\xi_i}(X) \), and thus \( Q \in \Sii{0}{\xi}(X) \).
Let
\begin{align*}
 Q^* =\big\{ (x, \gamma, i) \in X \times \l \times \l \mid {} & (x, \gamma, i) \in Q \wedge \forall i'<i \, \forall \gamma'<\l \, [\left(x, \gamma', i'\right) \notin Q^i] \\ 
 & {} \wedge \forall \gamma'<\gamma \, \left[\left(x, \gamma', i\right) \notin Q^i\right] \big\} .
\end{align*}
For any fixed pair $(\gamma, i) \in \l \times \cof(\xi)$, the set
\[ 
\left\{ x \in X  \mid \forall i'<i \, \forall \gamma'<\l  \left[\left(x, \gamma', i'\right) \notin Q^i\right] \wedge \forall \gamma'<\gamma \,\left[\left(x, \gamma', i\right) \notin Q^i\right] \right\} 
\]
belongs to \( \Pii{0}{\xi_i}(X) \).
Since \( \Pii{0}{\xi_i}(X) \subseteq \Sii{0}{\xi}(X) \)
and \( \Sii{0}{\xi} \) is \( \l \)-reasonable, $Q^* \in \Sii{0}{\xi}(X \times \l \times \l)$. 
As before, it follows that the set $R^*$ consisting of those $(x, \gamma) \in X \times \l$ such that $(x, \gamma, i) \in Q^*$ for some $i<\cof(\xi)$ is the desired uniformization of $R$ in $\Sii{0}{\xi}(X \times \l )$.
\end{proof}

By Corollary~\ref{separation_partition}, we thus get:

\begin{corollary}\label{separation_partition2} 
Let \( X \in \metr \) and  \( 1 < \xi < \l^+ \). 
Let \( C \in \Dee{0}{\xi}(X) \), and let \( \{ P_0, \dotsc, P_n \} \subseteq \Pii{0}{\xi}(X) \) be a finite family of pairwise disjoint subsets of \( C \). Then there is a \( \Dee{0}{\xi} \)-partition \( \{ C_0, \dotsc, C_n \} \) of \( C \) such that \( P_i \subseteq C_i \) for every \( i \leq n \).

The same is true for \( \xi = 1 \) if either \( \mu > \omega \), or \( \mu = \o \) and we restrict the attention to the subclass of \( \metr \) consisting of all ultrametrizable spaces.
\end{corollary}

\section{Generalized Borel functions as limits of continuous functions} \label{sec:Borelfunctions}

Let \( \boldsymbol{\Gamma}\) be a boldface pointclass.
 Let $X, Y \in \metr$,  and let \(\F\) be some set
of functions between X and Y. A function  \(f \colon X \to Y\) is \markdef{locally in \( \F \) on a \( \boldsymbol{\Gamma}\)-partition} \((A_\a)_{\a<\nu}\) of \( X \) if for each \( \a < \nu \) there is  \(f_\a \in \F \) such
that $f \restriction A_\a= f_\a \restriction A_\a$ for every $\a<\nu$. We will often consider functions which are locally constant on a \( \boldsymbol{\Gamma}\)-partition.

\begin{lemma} \label{lem:locallyconstant}
Let \( X , Y \in \metr \), and let \( \xi < \l^+ \) be a limit ordinal.
If \( f \colon X \to Y \) is locally constant on a finite \( \Dee{0}{\xi} \)-partition of \( X \), then 
\[ 
f \in \llim{\cof(\xi)} \M_{<\xi}(X,Y).
\]
\end{lemma}

\begin{proof}
Let \(n \in \o\) and \((A_j)_{j \leq n}\) be a finite \(\Dee{0}{\xi}\)-partition of \(X\) such that \( f \) is constant with value \( y_j \) on each \( A_j \). 
Using Lemma~\ref{lem:reg} if \( \mu > \o \) or \cite[Proposition 4.2.1]{DMR} if \( \mu = \o \), 
we can find a sequence of ordinals \((\xi_i)_{i<\cof(\xi)}\) cofinal in \(\xi\) and sets \(B^j_i \in \Dee{0}{\xi_i}(X) \) such that \( A_j=\bigcup_{i<\cof(\xi)}B^j_i\) for every \( j \leq n \), and moreover \( B^j_i\subseteq B^j_{i'}\) for every \(i \leq i' < \cof(\xi) \). 
Fix \(\bar y \in Y \), and for \( i <\cof(\xi)\) let
\[
f_i(x)= \begin{cases}
y_j, &  \text{ if } x \in B^j_i \text{ for some } j \leq n, \\
\bar y , & \text{ otherwise}.
\end{cases}
\]
Notice that \( f \) is well-defined because \( B^j_i \subseteq A_j \) and \(A_j \cap A_{j'} = \emptyset \) for every \( j \neq j' \). 
Since \( X \setminus \bigcup_{j \leq n} B^j_i \in \Dee{0}{\xi_i}(X) \),
each \(f_i\) is constant on a finite \(\Dee{0}{\xi_i}\)-partition, and hence \(\Sii{0}{\xi_i}\)-measurable. 
It remains to show that \(f=\lim_{i<\cof(\xi)}f_i\). 
Given \(x \in X\), let \(j \leq n\) and \( i < \cof(\xi) \) be such that \(x \in B^j_i\). 
Since the sequence \((B^j_i)_{i<\cof(\xi)}\) is increasing, \( x \in B^j_{i'} \) for every \( i' \geq i \), and thus \( f_{i'}(x) = y_j = f(x) \).
\end{proof}

\begin{remark}\label{rem:partition} 
For later use (see Definition~\ref{def_hat_limit_definable}), we note that the family of functions \( (f_i)_{i < \cof(\xi)} \) from the proof of Lemma~\ref{lem:locallyconstant} satisfies a stronger form of convergence to \( f \). More precisely, setting \( X_i = \bigcup_{j \leq n} B^j_i \) we get a covering
\((X_i)_{i<\cof(\xi)}\) of $X$ such that for every \( i \leq i' < \cof(\xi) \):
\begin{itemizenew}
\item 
$X_i \in \Sii{0}{\xi}(X)$;
\item 
\(X_i \subseteq X_{i'}\);
\item 
$f_{i'}(x)=f(x)$ for every $x \in X_i$. 
\end{itemizenew}
\end{remark}

We now consider limits over the directed set \( \Fin_\l =  ([\lambda]^{<\aleph_0}, {\subseteq}) \) of finite subsets of \( \l \), ordered by inclusion.

\begin{proposition} \label{prop:finlimit}
Let \( X , Y \subseteq {}^\mu \lambda \), and 
fix an ordinal \( 1 \leq \xi < \lambda^+ \). Then every \( f \in \M_{\xi+1}(X,Y) \) can be written as
\[
f = \lim\nolimits_{d \in \Fin_\lambda} f_d,
\]
where each \( f_d \colon X \to Y \) is locally constant on a finite \( \Dee{0}{\xi} \)-partition of \( X \).
\end{proposition}

\begin{proof}
Let \( \tr_Y \) be the tree of \( Y \). For  each \( s \in \tr_Y \), let \( \Nbhd_s(Y) = \Nbhd_s \cap Y \) and fix any \( y_s \in \Nbhd_s(Y) \). Let \( \{ P^s_\alpha \mid \alpha < \lambda \} \) be a family of nonempty \( \Pii{0}{\xi}(X) \)-sets such that \( f^{-1}(\Nbhd_s(Y)) = \bigcup_{\alpha < \lambda} P^s_\alpha \).

Since \( |\tr_Y| \leq \lambda \), we can clearly work with the directed set \( D = ([\tr_Y \times \lambda]^{< \aleph_0},{\subseteq}) \) instead of \( \Fin_\lambda \). Fix a nonempty \( d \in D \). Let \( S_d = \{ s \in \tr_Y \mid \exists \alpha < \lambda \, [(s,\alpha) \in d] \} \), and let
\( \beta_0, \dotsc, \beta_k \) enumerate the set \( \{ \leng (s) \mid s \in S_d \}  \) in increasing order, for the appropriate \( k \in \omega \). To simplify the notation, for \( s \in \tr_Y \) we let \( \lev (s) = j \) if and only if \( \leng(s) = \beta_j \), and for every \( i \leq k \) we let
\[	
s|_i = 
\begin{cases}
s \restriction \beta_i & \text{if } \leng(s) \geq \beta_i \\
s & \text{otherwise}.
\end{cases}
\]
Finally, for any \( j \leq k \) we let \( S_d^j = S_d^{j,0} \cup S_d^{j,1} \), where
\begin{align*}
S_d^{j,0}  & = \{ s|_j \mid s \in S_d \wedge \lev(s) \geq j \}  \qquad \text{and} \\
S_d^{j,1}  & = \{ s \in S_d \mid \lev(s) < j \wedge s \text{ maximal in } S_d \}.
\end{align*}
Clearly, \( S^j_d \) consists of pairwise incomparable sequences because of the maximality requirement in the definition of \( S_d^{j,1} \). Moreover, \( S_d^j \) is finite because so is \( d \). 
Notice also that if \( s \in S^j_d \), then \( s|_i \in S^i_d \) for every \( i \leq j \).
Finally, by definition \( S^k_d \) coincides with the set of all maximal elements of \( S_d \).

We build a collection of finite \( \Dee{0}{\xi} \)-partitions \( \mathcal{C}_j = \{ C^j_s \mid s \in S^j_d \} \) of \( X \) satisfying the following two conditions:
\begin{enumerate-(1)}
\item \label{cond1finlimit}
For every \( s \in S^{j,0}_d \), \( P^t_\alpha \subseteq C^j_s \) for every \( (t,\alpha) \in d \) with \( t \supseteq s \).
\item \label{cond2finlimit}
For every \( 0 < j \leq k \) and \( s \in S^j_d \), \( C^j_s \subseteq C^{j-1}_{s|_{j-1}} \).
\end{enumerate-(1)}
In particular, \( \mathcal{C}_j \) refines \( \mathcal{C}_{j-1} \).
Notice that condition~\ref{cond2finlimit} is equivalent to: \( C^j_s \subseteq C^i_{s|_i} \) for every \( i \leq j \leq k \) and \( s \in S^j_d \).

The construction is by recursion on \( j \leq k \). If \( j = 0 \), then \( S^j_d = S^{j,0}_d = \{ s |_ 0 \mid s \in S_d \} \). For every \( s \in S^{j,0}_d \), let \( P_s = \bigcup \{ P^t_\alpha \mid (t,\alpha) \in d \wedge  t \supseteq s \} \). The finite family \( \{ P_s \mid s \in S^{j,0}_d \} \) consists of pairwise disjoint \( \Pii{0}{\xi} \)-sets because \( P^t_\alpha \subseteq f^{-1}(\Nbhd_t(Y)) \subseteq f^{-1}(\Nbhd_{t|_0}(Y)) \) and \( \Pii{0}{\xi}(X) \) is closed under finite unions. Using Corollary~\ref{separation_partition2}, let \( \mathcal{C}_0 = \{ C^0_s \mid s \in S^{j,0}_d = S^j_d \} \) be any (finite) \( \Dee{0}{\xi} \)-partition of \( X \) separating the sets \( P_s \) from each other. It is clear that~\ref{cond1finlimit} holds by construction, while~\ref{cond2finlimit} needs not to be checked in this case. 

Assume now that \( j > 0 \), and that \( \mathcal{C}_{j-1} \) has already been defined. Fix any \( s \in S^j_d \). We distinguish two cases. If \( s \in S^{j,1}_d \), then \( s = s|_{j-1} \in S^{j-1}_d \) and we can set \( C^j_s = C^{j-1}_s \). With this choice, \ref{cond2finlimit} is trivially satisfied. The remaining case is when \( s \in S^{j,0}_d \).
Let \( \hat{S}^j_d = \{ s|_{j-1} \mid s \in S^{j,0}_d \} \), and notice that \( \hat{S}^j_d \subseteq S^{j-1,0}_d \subseteq S^{j-1}_d \). For each \( \hat{s} \in \hat{S}^j_d \), we repeat the argument from the basic case \( j = 0 \) but working within \( C^{j-1}_{\hat{s}} \) and considering only those \( s \in S^{j,0}_d \) such that \( s|_{j-1} = \hat{s} \). More precisely, for each such \( s \) let \( P_s = \bigcup \{ P^t_\alpha \mid (t,\alpha) \in d \wedge t \supseteq s \} \). By~\ref{cond1finlimit} applied to \( \hat{s} \), we get \( P_s \subseteq C^{j-1}_{\hat{s}} \). Moreover, the sets \( P_s \) are pairwise disjoint and they belong to \( \Pii{0}{\xi}(X) \). So by Corollary~\ref{separation_partition2} we can find a \( \Dee{0}{\xi} \)-partition \( \{ C^j_s \mid s \in S^{j,0}_d, s|_{j-1} = \hat{s} \} \) of \( C^{j-1}_{\hat{s}}\) separating the sets \( P_s \) from each other. It is clear that both~\ref{cond1finlimit} and~\ref{cond2finlimit} are satisfied by construction.

For each \( d \in D \), let \( f_d \colon X \to Y \) be the unique function which is locally constant on the finite \( \Dee{0}{\xi} \)-partition \( \mathcal{C}_k \) and assumes value \( y_s \) on each \( C^k_s \in \mathcal{C}_k \). 

\begin{claim} \label{claim:finlimit}
For every \( (s,\alpha) \in d \) and \( x \in P^s_\alpha \) there is \(  t \in \tr_Y \) such that \( t \supseteq s \) and \( f_d(x) = y_t \).
\end{claim}

\begin{proof}[Proof of the Claim]
Let \( t \) be the unique sequence in \( S^k_d \) such that \( x \in C^k_t \), so that \( f_d(x) = y_t \). Let  \( j = \min \{  \lev(s), \lev(t) \} \). Notice that  \( x \in P^s_\a \subseteq C^j_{s|_j} \) by~\ref{cond1finlimit}. Also, \( x \in C^j_{t|_j} \) by~\ref{cond2finlimit}. Therefore \( s|_j = t|_j \) because \( C^j_{s|_j} \cap C^j_{t|_j} \neq \emptyset \). It follows that \( s \) and \( t \) are compatible, and indeed \( s \subseteq t \) because \( t \), being an element of \( S^k_d \), is maximal in \( S_d \). 
\end{proof}

To conclude the proof, we just need to show that \( \lim_{d \in D} f_d = f \). Fix \( x \in X \) and any \( s \in \tr_Y \) such that \( f(x) \in \Nbhd_s(Y) \). Let \( \alpha < \lambda \) be such that \( x \in P^s_\alpha \), and set \( d = \{ (s,\alpha) \} \). Then Claim~\ref{claim:finlimit} entails that for every \( d' \supseteq d \) there is \( t \supseteq s \) such that \( f_{d'}(x) = y_t \in \Nbhd_t(Y) \), and since \( \Nbhd_t(Y) \subseteq \Nbhd_s(Y) \) we are done. 
\end{proof}

\begin{corollary} \label{cor:doublelimit}
Let \( X , Y \in \metr \), and further assume that \( \dim(X) = \dim(Y) = 0 \) if \( \mu = \o \). 
For every limit ordinal \(  \xi < \l^+ \),
\[
\M_{\xi+1}(X,Y)  \subseteq  \llim{\Fin_\l} \left( \llim{\cof(\xi)} \M_{<\xi}(X,Y) \right).
\]
\end{corollary}

\begin{proof}
By Theorems~\ref{theorem:sikorski} and~\ref{theorem:sikorski2}, we can assume that \( X,Y \subseteq \pre{\mu}{\l} \). 
Therefore it is enough to combine Proposition~\ref{prop:finlimit} with Lemma~\ref{lem:locallyconstant}.
\end{proof}

Given a family \( \D \) of directed sets and a collection of functions \( \F \) between topological spaces \( X \) and \( Y \), we say that \( \F \) is closed under \( \D \)-limits if for every \( D \in \D \) and every family \( (f_d)_{d \in D} \) of functions from \( \F \) we have \( \lim_{d \in D} f_d \in \F \) (whenever such limit exists). 
The \markdef{\( \D \)-closure} of \( \F \) is the smallest collection of functions which contains \( \F \) and is closed under \( \D \)-limits.

We consider the following families of directed sets:
\begin{itemizenew}
\item 
\(\D_0= \{\k\leq \l \mid \k \text{ regular} \} \cup \{ \Fin_\l \} \) 
\item \label{D_infi}
\(\D_\l=\{ D  \mid |D| \leq \l  \}\) 
\end{itemizenew}
Clearly, \( \D_0 \subsetneq \D_\l \).

\begin{theorem}\label{teo:D_limits}
Let $X, Y \in \metr$, and assume that \( \dim(X) = 0 \) if \( \mu = \o \). 
For every function \(f \colon X \to Y\), the following are equivalent:
\begin{enumerate-(1)}
\item\label{teo:D_limits1} 
\(f\) is \( \l^+ \)-Borel measurable;
\item\label{teo:D_limits2} 
\(f\) is in the \( \D_0 \)-closure of the collection of all continuous functions;
\item\label{teo:D_limits4} 
\(f\) is in the \( \D_\l \)-closure of the collection of all continuous functions.
\end{enumerate-(1)}
\end{theorem}

Clearly, in Theorem~\ref{teo:D_limits} we can replace \( \D_0 \) and \( \D_\l \) with any intermediate \( \D_0 \subseteq \D \subseteq \D_\l \).

\begin{proof}
\ref{teo:D_limits1} \(\Rightarrow \) \ref{teo:D_limits2} 
First assume that either \( \mu > \o \), or else \( \mu = \o \) and \( \dim(Y) = 0 \). Then we may assume, without loss of generality, that \( X,Y \subseteq \pre{\mu}{\l}\) by Theorems~\ref{theorem:sikorski} and~\ref{theorem:sikorski2}. We show by induction on \( \xi < \l^+ \) that \( \M_{\xi+1}(X,Y) \) is contained in the \( \D_0 \)-closure of \( \M_1(X,Y) \). The basic case \( \xi = 0 \) is trivial, so assume that $\xi>0$ and fix any \( f \in \M_{\xi+1}(X,Y) \).  
If \(\xi\) is a limit ordinal, then since \( \M_{<\xi}(X,Y) = \bigcup_{\xi' < \xi} \M_{\xi'+1}(X,Y)\) and \( \cof(\xi) \in \D_0 \) (because \( \cof(\xi) \leq |\xi| \leq \lambda \) is regular), it is enough to use Corollary~\ref{cor:doublelimit} and the inductive hypothesis. 
If instead \(\xi = \xi'+1 \) is a successor ordinal, then we use Proposition~\ref{prop:finlimit} and the inductive hypothesis applied to \( \xi' \): this works because if a function is locally constant on a finite \( \Dee{0}{\xi} \)-partition of \( X \), then it is trivially \( \Sii{0}{\xi'+1} \)-measurable. 

It remains to consider the case where \( \mu = \o \) but \( \dim(Y) \neq 0 \). First we prove the result for all \( \l^+ \)-Borel measurable functions with finite range, which are precisely the functions which are locally constant on a finite \( \Bor \)-partition.

\begin{claim} \label{claim:finiterange}
If \( f \colon X \to Y \) is locally constant on a finite \( \Bor \)-partition, then \( f \) is in the \( \D_0 \)-closure of \( \M_1(X,Y) \).
\end{claim}

\begin{proof}[Proof of the Claim]
Let \( Z = f(X)\).  When construed as a function from \( X \) to \( Z \), the map \( f \) is still \( \l^+ \)-Borel measurable. Since \( Z \) is finite, and hence discrete, then \( Z \in \metr \) and  \( \dim(Z) = 0 \), therefore we already know from what we proved above that \( f \) is in the \( \D_0 \)-closure of \( \M_1(X,Z) \) (as computed among functions from \( X \) to \( Z \)). But since \( Z \) is finite, the latter coincides with the \( \D_0 \)-closure of \( \M_1(X,Z) \) when viewed as a collection of functions from \( X \) to \( Y \): since \( \M_1(X,Z) \subseteq \M_1(X,Y) \), we are done.
\end{proof}

Consider now an arbitrary \( \l^+ \)-Borel measurable function \( f \colon X \to Y \).
Let \( \tau \) be the topology of \( Y \). By~\cite[Corollary 4.3.6]{DMR}, there is a topology \( \tau' \supseteq \tau \) on \( Y \) such that \( Y' = (Y,\tau') \in \M_\l \), \( \dim(Y') = 0 \), and \( \Bor(Y') = \Bor(Y) \). It follows that \( f \), viewed as a function from \( X \) to \( Y' \), is still \( \l^+ \)-Borel measurable, and thus it is \( \Sii{0}{\xi+1} \)-measurable for some \( \xi < \l^+  \). 
Since both \( X \) and \( Y' \) can now be construed as subspaces of \( \pre{\mu}{\l} \) by Theorem~\ref{theorem:sikorski2}, 
we can apply Proposition~\ref{prop:finlimit} and get a family of functions \( f_d \colon X \to Y '\) which are locally constant on a finite \( \Bor \)-partition and such that \( f = \lim_{d \in \Fin_\l} f_d \), where the limit is computed with respect to \( Y' \). Since \( \tau' \supseteq \tau \), we still have \( f = \lim_{d \in \Fin_\l} f_d \) if the limit is computed with respect to \( Y \), and clearly the functions \( f_d \) remain locally constant on the same finite \( \Bor \)-partition if we step back from \( Y' \) to \( Y \). Therefore we are done by Claim~\ref{claim:finiterange}.

\ref{teo:D_limits2} $\Rightarrow$ \ref{teo:D_limits4}  
Trivial, as \( \D_0 \subseteq \D_\l \).

\ref{teo:D_limits4} $\Rightarrow$ \ref{teo:D_limits1}
Let \(f=\lim_{d \in D}f_d\) with \((f_d)_{d \in D}\) a family of \( \l^+ \)-Borel measurable functions.  
Assume first that \( \mu > \o \), so that \( Y \) is zero-dimensional.
Then for every clopen \(U \subseteq Y\), 
\begin{equation} \label{eq:preimage1}
f^{-1}(U)=\bigcup_{d \in D} \bigcap_{d' \geq d }f^{-1}_{d'}(U).
\end{equation}
Since \(|D|\leq \l\) and \(f^{-1}_{d'}(U) \in  \Bor(X)\), this proves that \(f\) is \( \l^+ \)-Borel measurable too. 
If instead \( \mu = \o \), given any open set \( U \subseteq Y \) we consider an open covering \( \{ U_n \mid n \in \o  \} \) of \( U \) such that \( \cl(U_n) \subseteq U \) for every \( n \in \omega \). Then 
\begin{equation} \label{eq:preimage2}
f^{-1}(U)=\bigcup_{n \in \o} \bigcup_{d \in D} \bigcap_{d' \geq d}f^{-1}_{d'}(\cl(U_n)) , 
\end{equation}
hence \( f^{-1}(U)  \in \Bor(X) \) again.
\end{proof}

One might wonder whether the class \( \D_0 \) can be further reduced, still getting an analogue of Theorem~\ref{teo:D_limits}. For example, in the classical setting, which would correspond to the case \( \l = \omega \), it is enough to consider \( \o \)-limits (Theorem~\ref{thm:closureunderlimits-classical}). This is no longer true in the uncountable setup. 
For example, the following proposition implies that if \( \l \) has uncountable cofinality, then \( \M_{<\o}(X,Y) \), which is a proper subclass of all \( \l^+ \)-Borel measurable functions if \( X \) and \( Y \) are large enough, is already closed under \( \l \)-limits, and thus it contains the \( \D \)-closure of \( \M_1(X,Y) \) for \( \D = \{ \l \} \). (See also Corollary~\ref{cor:measurable_limits}.)

\begin{proposition}\label{prop:long_limits_counterex} 
Let \( X,Y \in \metr \). Let \( \k \leq \l \)
and \( \xi < \l^+  \) be such that \( \xi \) is limit and \( \cof(\xi) < \cof(\k) \). Then \(  \M_{<\xi}(X,Y) \) is closed under \( \kappa \)-limits.
\end{proposition}

\begin{proof} 
Suppose that \(f=\lim_{\a<\k}f_\a\) for some sequence \( (f_\a)_{\a < \k} \) of functions in \( \M_{<\xi}(X,Y)\).
Let \((\xi_i)_{i <\cof(\xi)}\) be a sequence of ordinals cofinal in \(\xi\).  Then for every \(\a<\k\) there exists \(i<\cof(\xi)\) such that \(f_\a\) is $\Sii{0}{\xi_i}$-measurable. Since \(\cof(\xi)< \cof(\k)\), there exists some \(\bar \imath <\cof(\xi)\) such that 
\[
A = \{ \a < \k \mid f_\a \text{ is \( \Sii{0}{\xi_{\bar \imath}} \)-measurable} \} 
\]
is unbounded in \( \k \), so that \( f = \lim_{\a \in A} f_\a \). Being a limit of \( \Sii{0}{\xi_{\bar \imath}}\)-measurable functions over an index set of size at most \( \l \), we get that \( f \) is \( \Sii{0}{\xi_{\bar \imath}+1} \)-measurable (by the computations~\eqref{eq:preimage1} and~\eqref{eq:preimage2} in Theorem~\ref{teo:D_limits}), and thus \( f \in \M_{< \xi}(X,Y) \) because \( \xi \) is limit.
\end{proof}

On the other hand, short sequential limits do not suffice either. 
Indeed, if \( \l \) is regular, then \( \Dee{0}{2}(X) \) is a \( \l \)-algebra. Therefore the class of \( \Dee{0}{2} \)-measurable functions, which is contained in \( \M_2(X,Y) \), is closed under \( D \)-limits for all directed sets \( D \) with \( |D| < \l \), and thus it already contains the \( \D \)-closure of \( \M_1(X,Y) \) for \( \D = \{ \k < \l \mid \k \text{ regular} \} \). 
It is open whether the class of \( \l^+ \)-Borel measurable functions can be realized as the \( \D \)-closure of \( \M_1(X,Y) \) for  \( \D =  \{ \k \leq \l \mid \k \text{ regular} \} \) or  \( \D = \{\Fin_\l \} \) (see Section~\ref{sec:questions} for more on the matter).

We conclude this section by showing that there is a variant of \( \Fin_\l \) which is rich enough to generate the whole class of \( \l^+ \)-Borel measurable functions by itself. The idea is to still consider finite subsets of \( \l \), but labeling each of their elements with an ordinal number. More precisely, for every \( d \in \pre{\l}{\l} \) let \( \supp(d) = \{ i < \l \mid d(i) \neq 0 \} \) be the support of \( d \), and let
\[  
\widehat \Fin_\l = \{ d \in \pre{\l}{\l} \mid \supp(d) \text{ is finite and } d(i) < 2+i \text{ for all } i < \l \},
\]
be ordered pointwise, that is, for all \(  d, {d}' \in \widehat{\Fin}_\l \) set 
\[
 d \leq   d' \iff  d(i) \leq  d'(i) \text{ for all } i < \l  .
\]
Clearly \( \widehat{\Fin}_\l \in \D_\l  \), and \( d \leq d' \Rightarrow \supp(d) \subseteq \supp(d') \).

\begin{lemma} \label{lem:unique}
For every \( D \in \D_0 \),
there is a surjection \( \iota \colon \widehat{\Fin}_\l \to D \) which is order-preserving, i.e.\ \( d \leq d' \Rightarrow \iota(d) \leq \iota(d') \) for every \( d,d' \in \widehat{\Fin}_\l \).
\end{lemma}

\begin{proof}
If \( D = \Fin_\l \), then we let \( \iota( d) = \supp(d) \).
If \( D = \k \) for some regular \( \k < \l \), we let \( \iota ( d) =  d(\kappa) \).
Finally, if \( D = \l \) then we let \( \iota(d) = \max \supp(d) \). It is easy to check that in all three cases \( \iota \) is as required.
\end{proof}

The map \( \iota \) from Lemma~\ref{lem:unique} allows us to simulate any \( D \)-limit with a \( \widehat{\Fin}_\l \)-limit, for every \( D \in \D_0 \).
Indeed, if \( f = \lim_{d \in D} f_d \), then \( f = \lim_{d \in \widehat{\Fin}_\l} \hat f_{d} \) once we set \( \hat f_{d} = f_{\iota(d)}\). Combining this with Theorem~\ref{teo:D_limits} we then get:

\begin{theorem} \label{teo:D_limitsunique}
Let \( X,Y \in \metr \), and assume that \( \dim(X) = 0 \) if \( \mu = \omega \). For every function \( f \colon X \to Y \), the following are equivalent:
\begin{enumerate-(1)}
\item 
\( f \) is \( \l^+ \)-Borel measurable;
\item 
\( f \) is in the \( \widehat{\Fin}_\l \)-closure of the collection of all continuous functions.
\end{enumerate-(1)}
\end{theorem}


\section{Generalized Baire class 1 functions and $\l$-full functions} \label{sec:baire_class_1}

The following definitions and results generalize~\cite[Section 2]{MR09} to the uncountable setup. 

\begin{definition} \label{definition_full_sets}
Let \( (X,d) \) be a \( \GG \)-ultrametric space. A set \( A  \subseteq X \) is \markdef{full} (\markdef{with constant} \( \rho \in \GG^+ \))
if \( B_d(x,\rho) \subseteq A \) for every \( x \in A \).
\end{definition}

Obviously, if \( \rho' \in \GG^+ \) is smaller than \( \rho \), then every set \( A \subseteq X \) which is full with constant \( \rho \) is also full with constant \( \rho' \).
If \( X = \pre{\mu}{\lambda} \), then \( B_d(x,r_\a) = \Nbhd_{x \restriction \a}\), and thus
a set \( A \subseteq \pre{\mu}{\l} \) is full with constant \( r_\alpha \) if and only if \( \Nbhd_{x \restriction \a} \subseteq A \) for every \( x \in A \).

\begin{proposition} 
Let \( (X,d) \) be a \( \GG \)-ultrametric space.
For every \( \rho \in \GG^+ \), the collection of all full subsets of \( X \) with constant \( \rho \) is a complete algebra consisting of clopen sets. Therefore, the collection of all full subsets of \( X \) (with any constant) is a \( \mu \)-subalgebra of its clopen sets.
\end{proposition}

\begin{proof} 
It is obvious that the collection of full sets with constant \( \rho \) is closed under arbitrary unions and consists of open sets, so it is enough to show that it is also closed under complements. Let \( A \subseteq X \) be full with constant \( \rho \), and let \( y \in X \setminus A \). Assume towards a contradiction that \( B_d(y,\rho) \cap A \neq \emptyset \), as witnessed by \( x \). Then \( B_d(x,\rho) = B_d(y,\rho) \) and hence \( y \in A \) by fullness, a contradiction.

Finally, let \( (A_\alpha)_{\alpha < \nu} \) for \( \nu < \mu \) be a family of full sets, and let \( \rho_\alpha \in \GG^+ \) be such that \( A_\alpha \) is full with constant \( \rho_\alpha \). Since \( \mu \) is regular and \( \GG \) has degree \( \mu \), there is \( \rho \in \GG^+ \) such that \( \rho \leq \rho_\alpha \) for all \( \alpha < \nu \). Then each \( A_\alpha \) is full with constant \( \rho \), and thus so is \( \bigcup_{\alpha < \nu } A_\alpha \).
\end{proof}

\begin{definition}\label{definition_full_functions}
Let \( (X,d) \) be a \( \GG \)-ultrametric space, \( Y \) be any set, and \( \nu \) be a cardinal. 
A function \( f \colon X \to Y \) is called \markdef{\( \nu \)-full} (\markdef{with constant \( \rho \in \GG^+ \)}) if \( |\ran(f)| \leq \nu \), and for every \( y \in \ran(f) \) its preimage \( f^{-1}(y) \) is full (with constant \( \rho \)).
The function \( f \) is \markdef{\( < \nu \)-full} (\markdef{with constant \( \rho \in \GG^+ \)}) if it is \( \nu' \)-full (with constant \( \rho \)) for some \( \nu' < \nu \).
Finally, \( f \) is \markdef{full} (\markdef{with constant \( \rho \in \GG^+ \)}) if it is \( \nu \)-full (with constant \( \rho \)) for some cardinal \( \nu \) or, equivalently, if \( f^{-1}(y) \) is full (with constant \( \rho \)) for all \( y \in \ran(f) \). 
\end{definition}

Equivalently, \( f \) is \( \nu \)-full (with constant \( \rho \)) if it is locally constant on a partition of \( X \) consisting of at most \( \nu \)-many full sets (with constant \( \rho \)). 
Note also that if \( f \) is \( < \nu \)-full for some \( \nu \leq \mu \), then
there is \( \rho \in \GG^+ \) such that \( f \) is \( < \nu \)-full \emph{with constant \( \rho \)}.  

As in the classical setting, full functions are intimately related to Lipschitz functions, where we say that a map \( f \) between two \( \GG \)-metric spaces \( (X,d_X) \) and \( (Y,d_Y) \) is \markdef{Lipschitz} (\markdef{with constant} \( R \in \GG^+ \)) if for all \( x,y \in X \)
\[  
d_Y(f(x),f(y)) \leq R \cdot d_X(x,y).
\]
(This makes sense because we assumed that \( \GG \) is a field.)

\begin{lemma}\label{fullarecont} 
Let \( (X,d_X) \) be a \( \GG \)-ultrametric space, and \( Y \) be a topological space. 
If \( f \colon X \to Y \) is full, then it is continuous.

Moreover, if \( (Y,d_Y) \) is a \( \GG \)-metric space, \( f \) has bounded%
\footnote{A subset \( A \) of a \( \GG \)-metric space \( (X,d) \) is \markdef{bounded} if there is \( r \in \GG^+\) such that \( \diam(A) \leq r \), i.e.\ \( d(x,y) \leq r \) for all \( x,y \in A \).} 
range,
and there is  some \( \rho \in \GG^+ \) such that \( f \) is full with constant \( \rho \), then  \( f \) is Lipschitz. 
\end{lemma}

The second part of the lemma applies to any \( < \mu \)-full function, and also to any
full function with constant \( \rho \) whenever \( Y \) has bounded diameter.

\begin{proof}
The first part of the lemma is obvious, so we only prove the second one.
Let \( r' \in \GG^+ \) be such that \( \diam(\ran(f)) \leq r' \), and set \( R = r' \cdot \rho^{-1} \). Pick any \( x,x' \in X \). If \( d_X(x,x') < \rho \), then \( d_Y(f(x),f(x')) = 0 \leq R \cdot d_X(x,x') \) because \( f \) is full with constant \( \rho \). 
If instead \( d_X(x,x') \geq \rho \), then
\[  
d_Y(f(x),f(x')) \leq r' = r' \cdot \rho^{-1} \cdot \rho \leq R \cdot d_X(x,x'). 
\]
Thus \( f \) is Lipschitz with constant \( R \).
\end{proof}

\begin{lemma}\label{lem:comp_full_lip} 
Let \( (X,d_X) \) and \( (Z,d_Z) \) be \( \GG \)-ultrametric spaces, and let 
\( Y \) be any set.
Let \(f \colon X \to Y \) be a \(\nu\)-full function, for some cardinal \( \nu \). 
If \(h \colon Z \to X \) is a Lipschitz function, then \(f \circ h \colon Z \to Y \) is \(\nu\)-full. 
Moreover, if \( f \) were \( \nu \)-full with constant \( \rho \in \GG^+\), then there is \( \rho' \in \GG^+ \) such that \( f \circ h \) is \( \nu \)-full with constant \( \rho' \).
The same is true if we replace $\nu$-full with $<\nu$-full.
\end{lemma}

\begin{proof}
Suppose that \( h \) is Lipschitz with constant \( R \in \GG^+ \).
It is enough to show that the preimage via \(h\) of a full set \(A \subseteq X\)
with constant \(\rho \in \GG^+\) is a full set with constant \(\rho' = \rho \cdot R^{-1}\). Indeed, let \(z \in Z \) be such that \(h(z) \in A \) and let \(z' \in Z\) be such that \(d_Z(z, z') < \rho' \). Then \(d_X(h(z), h(z')) \leq R \cdot d_Z(z, z') < R \cdot \rho \cdot R^{-1}= \rho\), hence \(h(z') \in A\). This shows that \(B_{d_Z}(z , \rho' ) \subseteq h^{-1}(A)\). 
\end{proof}

We now come to the problem of finding the ``right'' generalization of
the classical notion of a Baire class \( 1 \) function. When we move to cardinals \( \lambda > \omega \) and consider functions \( f \colon X \to Y \) between two spaces \( X , Y \in \metr \), we have two options: either we only consider \( \lambda \)-limits of continuous functions (i.e.\ the class \( \llim{\l} \M_1(X,Y) \)), or, in view of Theorem~\ref{teo:D_limits}, we allow limits over arbitrary directed sets of size at most \( \l \) (i.e.\ we consider \( \llim{\D_\l} \M_1(X,Y) \)). In this paper, the former are dubbed \markdef{\( \lambda \)-Baire class \( 1 \) functions}, while the latter are called \markdef{weak \( \l \)-Baire class \( 1 \) functions}. 
We are going to show that if \( \lambda \) is regular and \( Y \) is spherically complete, then the two notions (as well as all intermediate ones) coincide, and if moreover \( X \) is a \( \GG \)-ultrametric space, then this is the same as considering the class of \( \lambda \)-limits of \( \lambda \)-full functions.

\begin{proposition}\label{full_baire} 
Let \( \lambda \) be regular, and let \( X, Y \subseteq \pre{\l}{\l} \) with \( Y  \) superclosed.
For every $\Sii{0}{2}$-measurable function $f \colon X \to Y$ there is a sequence of functions \( (f_\alpha)_{\alpha < \lambda }\) such that \( f  = \lim_{\a < \l} f_\a \) and \( f_\alpha \colon X \to Y \) is \( \lambda \)-full with constant \( r_\alpha \). 

If moreover we assume $\lambda$ to be strong limit (hence inaccessible), then each \( f_\alpha \) can be taken to be $<\lambda$-full with constant \( r_\a \).
\end{proposition}

\begin{proof} 
For \( i, \ell < \lambda \), set $\mathcal{N}^\ell_{i} =\{y \in Y \mid y(i)=\ell\} $. Since \( \mathcal{N}^\ell_i \) is clopen, $f^{-1}(\mathcal{N}^\ell_{i}) \in \Sii{0}{2}(X)$. 
Let \( F^\ell_{j,i}\) be closed subsets of \( X \) such that  $f^{-1}(\mathcal{N}^\ell_{i})= \bigcup_{j<\lambda}F^\ell_{j,i}$. 
For every $i,j,\ell<\lambda$, let $\tr^\ell_{j,i}$ be the pruned tree of $ F^\ell_{j,i}$, and let $\tr_X$ be the pruned tree of $X$. 
Note that 
$\bigcup_{\ell<\lambda}f^{-1}(\mathcal{N}^\ell_{i})= X$ for every $i<\lambda$, and therefore
\begin{equation} \label{aster} 
\bigcup_{j,\ell<\lambda}\tr^\ell_{j,i} = \tr_X.
\end{equation}
For every $s \in  \tr_X$ and $i<\lambda$, 
let $(j^s_{i},\ell^s_{i} )$ 
be the smallest (with respect to the G\"odel ordering) pair \( (j,\ell) \in \lambda \times \lambda \)  such that $s \in \tr^{\ell}_{j,i}$. 
Notice that for every $x \in X$ and every $i<\l$ there is \( \xi^x_i < \l \) such that 
\( \ell^{x \restriction \a}_i = \ell^{x \restriction \xi^x_i}_i \) and \( j^{x \restriction \a}_i = j^{x \restriction \xi^x_i}_i \), for all \( \xi^x_i \leq \a < \l\).
(Otherwise \( x \notin \tr^\ell_{j,i}\) for any \( j,\ell < \l\), contradicting~\eqref{aster}.)
In particular, 
$\ell^{x\restriction \a}_{i}=f(x)(i)$ for every \( \a \geq \xi^x_i \). 

Let $\tr_Y$ be the superclosed pruned tree of $Y$, and for every sequence \(t \in \tr_Y\) pick some \(y_t \in Y\) such that \(t \subseteq y_t\).
Let \( \phi' \colon \tr_X \to \pre{<\l}{\l}\) be such that \( \leng(\phi^\prime (s)) = \leng(s) \) and \( \phi^\prime (s)(i) = \min \{ \ell^s_i,\leng(s) \} \) for every \( i < \leng(s) \). 
We define the map \( \phi \colon \tr_X \to \tr_Y\)  by letting \(\phi(s) \subseteq \phi^\prime(s)\) be the longest sequence still in \(\tr_Y\), namely, 
\[
\phi(s) = 
    \bigcup \{v \in \tr_Y \mid v \subseteq  \phi^\prime(s) \}.
\] 
Note that \(\phi(s) \in \tr_Y\)  since $\tr_Y$ is superclosed.
Finally, for every \( \a < \l \), let \( f_\a \colon X \to Y \) be defined by
\[  
f_\a(x) = y_{\phi(x \restriction \a )}.
\]

Notice that \( \varphi'(s) \in \pre{\leng(s)}{(\leng(s)+1)} \), thus \( f_\a \) attains at most \( |\pre{\a}{(\a+1)}| \)-many values. This means that \( |\ran(f_\a)| \leq \l \), and furthermore \( |\ran(f_\a)| < \l \) when \( \l \) is strong limit.
Moreover, if \( x \restriction \a = y \restriction \a \), then \( f_\a(x) = f_\a(y) \). Therefore, \( f_\a \) is \( \l \)-full (or even \( < \l \)-full, if \( \l \) is strong limit) with constant \( r_\a \). 

It remains to show that \( f  = \lim_{\a < \l } f_\a \).
To this aim, fix an element $x \in A$: we need to prove that for every $\g<\l$ 
there exists $\bar\a<\l$ such that $f_\a(x)\restriction \g = f(x) \restriction \g $  for every $\a \geq \bar\a$. Set 
\[
\bar\a =  \sup \{\g, \xi^x_i, f(x)(i) \mid i < \g \}.  
\]
Note that $\bar\a<\l$ because $\l$ is regular. Fix any \( i < \g \) and \( \a \geq \bar \a \). Since \( \bar \a \geq \xi^x_i \), then \( \ell^{x \restriction \a}_i = f(x)(i) \). Since \( \bar \a \geq f(x)(i) = \ell^{x \restriction \a}_i\), then \( \phi^\prime(x \restriction \a)(i) = \ell^{x \restriction \a}_i \); this implies that  \( \phi^\prime(x \restriction \a) \restriction \g =f(x) \restriction \g \in \tr_Y \), hence \( \phi(x \restriction \a) (i) =  \phi^\prime(x \restriction \a)  (i) \) . Finally, since \( \bar \a \geq \gamma \), then \( f_\alpha(x)(i) = \phi(x \restriction \a )(i) \).
It follows that $f_\a(x)\restriction \g = f(x) \restriction \g $, as desired.
\end{proof}

The following is the analogue in the uncountable regular case of~\cite[Corollary 2.16]{MR09}.

\begin{theorem}\label{car_baireclass1} 
Let \( \l \) be regular, \( X , Y \in \metr \), and suppose that \( Y \) is spherically complete. 
For every $f \colon X \to Y$, the following are equivalent:
\begin{enumerate}[label={\upshape (\arabic*)}, leftmargin=2pc]
\item\label{I} 
$f$ is a \( \l \)-Baire class \( 1 \) function;
\item\label{II} 
$f$ is a weak \( \l \)-Baire class \( 1 \) function; 
\item\label{III} 
$f$ is $\Sii{0}{2}$-measurable.
\end{enumerate}
If \( X \) is a \( \GG \)-ultrametric space, then the above conditions are also equivalent to
\begin{enumerate}[label={\upshape (\arabic*)}, leftmargin=2pc,resume]
\item\label{IIII} 
$f = \lim_{\a < \l } f_\a$ with \( f_\a \) a \( \l \)-full function with constant \( \rho_\a \), for some \( \rho_\a \in \GG^+ \).
\end{enumerate} 
In case \( \l \) is strong limit (and \( X \) is \( \GG \)-ultrametric), item~\ref{IIII} can be replaced by
\begin{enumerate}[label={\upshape (4\('\))}, leftmargin=2pc]
\item\label{IIII'} 
$f$ is a $\l$-limit of  $<\l$-full functions,
\end{enumerate}
and if \( Y \) is a \( \GG \)-metric space we can further add: 
\begin{enumerate}[label={\upshape (5)}, leftmargin=2pc]
\item\label{V} 
\( f \) is a \( \l \)-limit of  Lipschitz functions.
\end{enumerate}
\end{theorem} 

Notice that because of Lemma~\ref{fullarecont}, condition~\ref{V} can be added to the list also when \( X \) is \( \GG \)-ultrametric and \( Y \) is a \( \GG \)-metric space \emph{with bounded diameter}, independently of whether \( \l \) is strong limit or not.

\begin{proof}
The implication \ref{I} $\Rightarrow$ \ref{II} is obvious. To see \ref{II} $\Rightarrow$ \ref{III}, suppose that \( f \) is the \( D\)-limit of a family \( (f_d)_{d \in D} \) of continuous functions \( f_d \colon X \to Y \), for some \( D \in \D_\l \). Since \(Y\) is zero-dimensional, it is enough to check that \( f^{-1}(U) \in \Sii{0}{2}(X) \) for any clopen set \( U \subseteq Y \), which is the case because \( f^{-1}(U) = \bigcup_{ d \in D } \bigcap_{ d' \geq d } f_{d'}^{-1}(U) \).
To prove \ref{III} \( \Rightarrow \) \ref{I}, first notice that by Theorem~\ref{theorem:sikorski}  we can assume that \( X,Y \subseteq \pre{\l}{\l} \) with \( Y \) superclosed: applying Proposition~\ref{full_baire} and using Lemma~\ref{fullarecont}, we get the desired result.

Suppose now that \( (X,d) \) is a \( \GG \)-ultrametric space. 
Without loss of generality, we can again assume that \( Y \) is a superclosed subset of \( \pre{\l}{\l} \).
By Proposition~\ref{prop:bilip} there is a topological embedding \( h \colon X \to \pre{\l}{\l} \) which is Lipschitz. Let \( X' = \ran(h) \). Since the map \( f \circ h^{-1}\) is \( \Sii{0}{2}\)-measurable, by Proposition~\ref{full_baire}, it is the \( \l \)-limit of a sequence \( (f'_\a)_{\a < \l} \) of \( \l \)-full (or \( < \l \)-full, if \( \l \) is strong limit) functions \( f'_\a \colon X' \to Y \) with constant \( r_\a \). Then  by Lemma \ref{lem:comp_full_lip} each \( f_\a = f'_\a \circ h \colon X \to Y \) is \( \l \)-full (or \( < \l \)-full) with constant \( \rho_\a \) for some \( \rho_\a \in \GG^+ \), and \( f  = \lim_{\a < \l} f_\a\). This proves~\ref{III} \( \Rightarrow \) \ref{IIII}, and also~\ref{III} \( \Rightarrow \) \ref{IIII'} if \( \l \) is strong limit. 
Lemma~\ref{fullarecont} yields \ref{IIII} \( \Rightarrow \) \ref{I} and \ref{IIII'} \( \Rightarrow \) \ref{I}. 

Finally, assume that \( \l \) is strong limit, \( X \) is \( \GG \)-ultrametric, and
 \( Y \) is a \( \GG \)-metric space.
Every \( < \l \)-full function from \( X \) to \( Y \) has bounded range because the order of the field \( \GG \) has cofinality \( \l \), therefore \ref{IIII'} \( \Rightarrow \) \ref{V} by Lemma~\ref{fullarecont} again. 
Finally, \ref{V} \( \Rightarrow \) \ref{I} because Lipschitz functions are obviously continuous.
\end{proof}

\section{Level-by-level analysis of generalized Borel functions} \label{sec:long}

The goal of this section is to isolate the ``right'' notion of Baire class \( \xi \) function in the generalized context. More in detail, we seek a stratification of \( \l^+ \)-Borel measurable functions defined in terms of suitable limit operators, which starts from continuous functions and is ``compatible'' with \( \Sii{0}{\xi} \)-measurability in the sense of Theorem~\ref{thm:Bairestratification-classical}. 
We first deal with regular cardinals, as in this case we can exploit the results of Section~\ref{sec:baire_class_1} through a suitable change-of-topology argument. After that, we briefly discuss what needs to be modified in the singular case in order to get analoguous definitions and results.

\subsection{The regular case} \label{subsec:longregular}

The results from Section~\ref{sec:baire_class_1} allow us to obtain structural information on the stratification in terms of measurability of $\l^+$-Borel functions.
Indeed, using also Lemma~\ref{lem:changetopnew} we can prove that if $\xi$ is a successor ordinal or a limit ordinal of cofinality $\l$, then $\Sii{0}{\xi+1}$-measurable functions can be characterized as $\l$-limits of functions from lower levels (see Theorem~\ref{thm:measurable_limits}).

\begin{lemma}\label{lem:changetopnew}
Let \(\l\) be regular, \( X  \in \metr \), and  \(1 \leq \xi<\l^+\). For every family \((A_i)_{i<\l}\) of \(\Sii{0}{\xi}\)-subsets of \(X\) there exist \(D \subseteq \pre{\l}{\l}\) and a bijection \(h \colon D  \to X\) such that:
\begin{enumerate-(1)-r}
\item \label{lem:changetop1} 
\(h\) is continuous;
\item \label{lem:changetop2}
for every \( s \in \pre{\l}{\l}\), one of the following holds, depending on the given \( \xi \):
\begin{description}
\item[Case 1]
if \(\xi = \xi'+1 > 1\) is successor, then
\(h(\Nbhd_s \cap D) \in \Pii{0}{\xi'} \);
\item[Case 2]
if \(\xi\) is limit with \(\cof(\xi) = \l\), then
\(h(\Nbhd_s \cap D) \in \Pii{0}{\nu(\leng(s))} \), for some fixed non-decreasing  function \(\nu \colon \l \to \xi\);
\item[Case 3]
if \(\xi\) is limit with \( \cof(\xi) < \l\), then
\(h(\Nbhd_s \cap D) \in \Pii{0}{\xi} \);
\end{description}
\item \label{lem:changetop3} 
\(h^{-1}(A_i) \in \Sii{0}{1}(D)\) for every \(i<\l\).
\end{enumerate-(1)-r}
\end{lemma}

\begin{proof}
Without loss of generality, we can assume that \(X \subseteq \pre{\l}{\l}\) by Theorem~\ref{theorem:sikorski}.
Let \((C_{\g})_{\g<\l} \) be a family of sets in \(\bigcup_{\xi'<\xi}\Pii{0}{\xi'}\) such that  \(A_i=\bigcup_{\g \in K_{i}}C_{\g}\) for some \(K_{i} \subseteq \l\). 

We construct a family \(\{B_s \mid s \in \pre{<\l}{\l}\}\) of subsets of \(X\) such that for every \( \a,\g < \l \) and \( s \in \pre{<\l}{\l}\):
\begin{enumerate-(i)}
\item\label{lem:1} 
\( B_s \subseteq B_t\) whenever \( t \subseteq s \);
\item\label{lem:2} 
\( \{ B_t \mid \leng(t) = \a \} \) is a disjoint covering of \( X \);
\item\label{lem:3} 
\(B_s \subseteq \Nbhd_{t_s} \cap X\) for some \(t_s \in \pre{<\l}{\l}\) with \(\leng(t_s)=\leng(s)\);
\item\label{lem:4}    
\(B_s \in \Pii{0}{\xi'}\) if \(\xi = \xi'+1>1\) is a successor ordinal;  \(B_s \in \Pii{0}{\nu(\leng(s))}\) for a suitable non-decreasing map \(\nu \colon \l \to \xi\) if \(\xi\) is limit ordinal of cofinality \(\l\); \(B_s \in \Pii{0}{\xi}\) if \(\xi\) is a limit ordinal of cofinality smaller than \( \l\);
\item\label{lem:5}  
\(C_{\g}=\bigcup_{t \in K}B_t\)  for some \(K \subseteq \pre{\g+1}{\l}\).
\end{enumerate-(i)}

Consider the induced map $h \colon D \subseteq \pre{\l}{\l} \to X$, \(x \mapsto h(x) \in \bigcap_{i<\l}B_{x \restriction i} \), where \(D=\{ x \in \pre{\l}{\l} \mid  \bigcap_{i<\l}B_{x \restriction i} \neq \emptyset \}\). 
Note that \(h\) is well-defined and continuous (so that~\ref{lem:changetop1} is satisfied) by~\ref{lem:1} and~\ref{lem:3}. Condition~\ref{lem:2} entails that \( h \) is a bijection and that \( h(\Nbhd_s \cap D) = B_s \), so that~\ref{lem:changetop2} is satisfied as well by~\ref{lem:4}.
Finally, observe that by~\ref{lem:5} and~\ref{lem:2} applied to \( \a = \g+1 \),
\begin{align*}
h^{-1}(C_{\g}) &=\bigcup\nolimits_{t \in K}h^{-1}(B_t)=\bigcup\nolimits_{t \in K}(\Nbhd_t \cap D) \\
h^{-1}(X \setminus C_{\g})&=\bigcup\nolimits_{t \in \pre{\g+1}{\l}\setminus K} h^{-1}(B_t) = \bigcup\nolimits_{t \in \pre{\g+1}{\l} \setminus K}(\Nbhd_t \cap D).
\end{align*}
Hence \( h^{-1}(C_\g ) \) is clopen for every \( \g < \l \), so that~\ref{lem:changetop3} follows by the choice of the sets \( C_\g \).

The construction of the family \(\{B_s \mid s \in \pre{<\l}{\l}\}\) is by induction on \(\a=\leng(s) <\l\); when \( \xi \) is a limit ordinal of cofinality \( \l \), along the recursive process we also define the value \( \nu(\a) \) of the function \( \nu \colon \l \to \xi\) on input \( \a \). We set \(B_\emptyset=X\) and, if \( \xi \) is limit of cofinality \( \l \), \( \nu(0) = 1 \).

Suppose first that \(\a=\g+1\) is a successor ordinal. 
Let \( \bar \xi < \xi \) be such that \( C_\g \in \Pii{0}{\bar \xi} \).
By Lemma~\ref{lem:reg} (applied to \( X \setminus C_\g \)),
there are \( J \leq \l \) and a \( \Pii{0}{\xi} \)-partition \( \C_\g = \{ C^{(j)}_\g \mid j < J \}\) of \( X \) such that \( C^{(0)}_\g = C_\g \). 
For each \( t \in \pre{\a}{\l} \), \( j < J \), and \( s \in \pre{\g}{\l} \), let \( P_{t,j,s} = \Nbhd_t \cap C^{(j)}_\g \cap B_s\). 
By inductive hypothesis, \( \{ B_s \mid s \in \pre{\g}{\l} \} \) is a disjoint covering of \( X \), hence so is 
\[
\P=\{ P_{t,j,s} \mid t \in \pre{\a}{\l}, j < J, s \in \pre{\g}{\l}\}.
\]
For every  \(s \in \pre{\g}{\l}\), 
fix a bijection \( f_s \colon \lambda \to \pre{\a}{\l} \times J \times \{ s \} \) and set \( B_{s \conc \b} = P_{f_s(\b)} \) for all \( \b < \l \).
We know check that all of~\ref{lem:1}--\ref{lem:5} are satisfied. 
Condition~\ref{lem:1} is obvious, while condition~\ref{lem:2} follows from the fact that \( \{ B_t \mid \leng(t) = \a \} = \P \).
Condition~\ref{lem:3} holds by construction, as it is enough to let \( t_{s \conc \b} \) be the first coordinate of \( f_s(\b) \), for every \( s \in \pre{\g}{\l} \) and \( \b < \l \).
We now check~\ref{lem:4}. 
If $\xi = \xi'+1>1$ is a successor ordinal, then by inductive hypothesis \(B_s \in  \Pii{0}{\xi'} \) for every \( s \in \pre{\g}{\l} \), and since \(C^{(j)}_\g \in \Pii{0}{\bar \xi }\) and \( \bar \xi \leq \xi' \), we conclude \(B_{s \conc \b} \in \Pii{0}{\xi'}\). 
If \(\xi\) is a limit ordinal of cofinality \(\l\), set \( \nu(\a)=\max\{\nu(\g), \bar \xi \}<\xi \). By inductive hypothesis again, \(B_s \in  \Pii{0}{\nu(\g)} \subseteq \Pii{0}{\nu(\a)} \) for every \( s \in \pre{\g}{\l }\), and \(C^{(j)}_\g \in \Pii{0}{\bar \xi} \subseteq \Pii{0}{\nu(\a)}\): hence \( B_{s \conc \b} \in \Pii{0}{\nu(\a)} \).
If instead \(\xi\) is a limit ordinal of cofinality smaller than \( \l\), then by inductive hypothesis \(B_s \in  \Pii{0}{\xi} \) and  \(C^{(j)}_\g \in \Pii{0}{\bar \xi} \subseteq \Pii{0}{\xi}\), hence \(B_{s \conc \b} \in \Pii{0}{\xi}\).
Finally, using that \( \P \) is a disjoint covering of \( X \), it is easy to verify that~\ref{lem:5} holds by taking \(K=\{s \conc \b \mid s \in \pre{\g}{\l} \wedge  B_{s \conc \b}\subseteq C^{(0)}_\g\}\). 

Now let \(\a>0\) be a limit ordinal. 
For every \(s \in \pre{\a}{\l}\), set \(B_s=\bigcap_{\b<\a}B_{s \restriction \b}\).
Then~\ref{lem:1} is satisfied by construction, while~\ref{lem:2} holds because by inductive hypothesis it is satisfied at all levels \( \b < \a \). Condition~\ref{lem:3} is satisfied by setting \( t_s = \bigcup_{\b < \a} t_{s \restriction \b } \).
To prove~\ref{lem:4} we again consider various cases. 
If \(\xi = \xi' + 1\) is a successor ordinal, then by inductive hypothesis \(B_{s \restriction \b} \in \Pii{0}{\xi'}\) for every \(\b<\a\), hence \(B_s=\bigcap_{\b<\a}B_{s \restriction \b} \in \Pii{0}{\xi'}\) as well. 
If \(\xi\) is a limit ordinal of cofinality \(\l\), then we set \( \nu(\a) = \sup_{\b < \a} \nu(\b) \): since \( \cof(\xi) = \l \), we still have \( \nu(\a) < \xi \). Then by inductive hypothesis \(B_{s \restriction \b} \in \Pii{0}{\nu(\b)} \subseteq \Pii{0}{\nu(\a)} \), hence \(B_s =\bigcap_{\b<\a}B_{s \restriction \b} \in \Pii{0}{\nu(\a)}\) too. Finally, if \(\xi\) is a limit ordinal with \( \cof(\xi) <\l\), then \( B_{s \restriction \b} \in \Pii{0}{\xi} \) for every \( \b < \a \), and thus \(B_s=\bigcap_{\b<\a}B_{s \restriction \b} \in \Pii{0}{\xi}\) as well.
Condition~\ref{lem:5} is relevant only at successor stages, and hence our proof is complete.
\end{proof}

\begin{corollary}\label{cor_changetopnew} 
Let \(\l\) be regular, \((X,\tau) \in \metr\), and \( 1 \leq \xi<\l^+\). 
Given a family \((A_i)_{i<\l}\) of \(\Sii{0}{\xi}(\tau)\) subsets of \(X\),  there is a topology \(\tau'\) on X such that:
\begin{enumerate-(1)}
\item 
\((X,\tau') \in \metr \); 
\item 
\(\tau \subseteq \tau' \subseteq \Sii{0}{\xi}(\tau)\) if \( \xi \) is either a successor ordinal or a limit ordinal with cofinality \( \l \), and \(\tau \subseteq \tau' \subseteq \Sii{0}{\xi+1}(\tau)\) otherwise;
\item \label{cor_changetop3}
\(A_i \in \Sii{0}{1}(\tau')\) for every \(i<\l\).
\end{enumerate-(1)}
\end{corollary}

\begin{proof}
Let \( h \colon D \to X \) be as in Lemma~\ref{lem:changetopnew}. Then
it is enough to let \(\tau'\) be the pushforward via \(h\) of the topology induced by \(\pre{\l}{\l}\) on \( D \).
\end{proof}

\begin{theorem}\label{thm:measurable_limits}
Let \( \l \) be regular, \( X , Y \in \metr \), and suppose that \( Y \) is spherically complete. Let  \( 1 \leq \xi < \l^+ \) be either a successor ordinal or a  limit ordinal of cofinality \(\l\).
Then
\[
\M_{\xi+1}(X,Y) = \llim{\l} \M_\xi(X,Y).
\]
If \(\xi\) is limit, then we further have 
\[
\M_{\xi+1}(X,Y) = \llim{\l} \M_{<\xi}(X,Y).
\]
\end{theorem}

\begin{proof} 
For the right-to-left inclusions, consider \( f  = \lim_{\a < \l} f_\a \) with \( f_\a \in \M_{\xi}(X,Y) \) or \( f_\a \in \M_{< \xi}(X,Y) \), depending on whether \( \xi \) is successor or limit. Fix a clopen basis \((U_i )_{i<\l}\) for \( Y \): it is enough to show that \( f^{-1}(U_i) \in \Sii{0}{\xi+1}(X) \) for every \( i < \l \).  As \( U_i \) is clopen, we have
\[  
f^{-1}(U_i) = \bigcup_{\bar \a < \l } \bigcap_{\bar \a \leq \a < \l} f_{\a}^{-1}(U_i),
\]
and since in any case \( f^{-1}_\a(U_i) \in \Dee{0}{\xi}(X) \), 
the result easily follows.

We now deal with the inclusion from left to right. 
Let \(\tau\) be the topology of \(X\).
The proof is by induction on $1 \leq \xi< \l^+$. The base case $\xi=1$ is Theorem \ref{car_baireclass1}, therefore we can move to the case $\xi>1$. 

Assume first that $\xi$ is a successor ordinal: in this case, we want to show that $f$ is a \(\l\)-limit of $\Sii{0}{\xi}$-measurable measurable functions \(f_\a \colon  X \to Y \). 
As \( f \in \M_{\xi+1}(X,Y) \), for every open set $U\subseteq Y$ we can write $f^{-1}(U)=\bigcup_{i<\l} A_{i}$ with $A_{i} \in \Pii{0}{\xi}(X)$. 
Applying Corollary~\ref{cor_changetopnew} to the family $(X \setminus A_{i})_{i<\l}$, we obtain a topology $\tau'$ on $X$ such that \( (X,\tau') \in \metr\), \(\tau \subseteq \tau' \subseteq \Sii{0}{\xi}(\tau)\), and \(A_{i} \in \Pii{0}{1}(\tau')\) for every \(i<\l\). 
Then $f$ is $\Sii{0}{2}(\tau')$-measurable: applying Theorem~\ref{car_baireclass1}, we get that \( f = \lim_{\a < \l} f_a \) with each \( f_\a \colon (X,\tau') \to Y \) continuous. Since \(\tau^\prime \subseteq \Sii{0}{\xi}(\tau)\), the functions \( f_\a \) are \( \Sii{0}{\xi} \)-measurable as maps between \( (X,\tau) \) and \( Y \), hence we are done.

Now assume that $\xi$ is a limit ordinal with $\cof(\xi)=\l$:
we need to show that $f$ is a $\l$-limit of a sequence of functions \(f_\a \colon  X \to Y \), where each \(f_\a\) is \(\Sii{0}{\xi_\a}\)-measurable for some \(\xi_\a<\xi\). Without loss of generality, we may assume that \( Y \) is a superclosed subspace of \( \pre{\l}{\l} \) (Theorem~\ref{theorem:sikorski}).
Fix a basis $(U_j)_{j<\l}$ of $Y$, and let \( \{ A_{i} \mid i < \l \} \subseteq \Sii{0}{\xi}(X) \) be such that for every \( j < \l \), \( f^{-1}(U_j) = \bigcup_{i \in K_j} (X \setminus A_i) \) for some \( K_j \subseteq \l \). 
By Lemma~\ref{lem:changetopnew}, there are a non-decreasing function \( \nu \colon \l \to \xi \), a subspace \( D \subseteq \pre{\l}{\l} \), and a continuous bijection \(h \colon D \to X \) such that \(h^{-1}(A_i) \in \Sii{0}{1}(D)\) for every \(i<\l\), and moreover
\(h(\Nbhd_s \cap D) \in \Pii{0}{\nu(\leng(s))} \) for every \(s \in \pre{<\l}{\l}\).
By the choice of the sets \( A_i \),
the function \(f \circ h \colon D \to Y\) is thus \(\Sii{0}{2}(D)\)-measurable.
By Proposition~\ref{full_baire}, there is a family of functions \( (g_\a)_{\a < \l } \) such that \( f \circ h = \lim_{\a < \l} g_\a \) and each \( g_\a \colon D \to Y \) is $\l$-full with constant $r_\a$: set \( f_\a = g_\a \circ h^{-1} \).
Clearly, \( f = \lim_{\a < \l} f_\a  \). Since \( g_\a \) is \( \l \)-full with constant \( r_\a \), the range of \( g_\a \) consists of (at most) \( \l \)-many points \( \{ y_i \mid i < \l \} \subseteq Y \) such that for every \( i < \l \) 
\[ 
g_\a^{-1}(y_i) = \bigcup \{ \Nbhd_t \cap D \mid t \in L_i \},
\]
for some \( L_i \subseteq \pre{\a}{\l} \). By the choice of \( h \), it follows that \( f_\a^{-1}(y_i) = h(g_\a^{-1}(y_i)) \in \Sii{0}{\nu(\a)+1}(X) \). Since \( \ran(f_\a) = \ran(g_\a) \) has at most \( \l \)-many elements, \( f_\a \) is thus \( \Sii{0}{\xi_\a} \)-measurable for \( \xi_\a = \nu(\a)+1 < \xi \).
\end{proof}

\begin{corollary} \label{cor:measurable_limits}
Let \( \l \) be regular, \( X,Y \in \metr \), and suppose that \( Y \) is spherically complete. Then the closure under \( \l \)-limits of the collection of all continuous functions is precisely \( \M_{< \o}(X,Y) \).

More generally, for every \( 1 \leq \xi < \l^+ \) which is either a successor ordinal, or a limit ordinal of cofinality \( \l \), the closure of \( \M_\xi(X,Y) \) under \( \l \)-limits is \( \M_{< \xi+\o}(X,Y) \).
\end{corollary}

\begin{proof}
Exploiting Theorem~\ref{thm:measurable_limits},
an easy induction shows that \( \M_{<\xi+\o}(X,Y) = \bigcup_{n \in \omega} \M_{\xi+n}(X,Y) \) is contained in the closure under \( \l \)-limits of \( \M_\xi(X,Y) \) . Since \( \cof(\xi+ \o) = \o < \l = \cof(\l) \), the reverse inclusion follows from Proposition~\ref{prop:long_limits_counterex}.
\end{proof}

Theorem~\ref{thm:measurable_limits} tells us that, in order to find an appropriate definition of generalized Baire class functions, we can use \( \l \)-limits at all successor stages and at all limit levels with cofinality \( \l \).
When \( \xi \) is a limit ordinal of cofinality smaller than \( \l \), instead, functions in \( \M_{\xi+1} \) can no longer be characterized as \( \l \)-limits of functions in \(  \M_{< \xi}(X,Y) \) because of Proposition~\ref{prop:long_limits_counterex}, so a different approach is in order. In this situation, Corollary~\ref{cor:doublelimit} suggests to use limits of the form
\[ 
\lim\nolimits_{\Fin_\l} \circ \lim\nolimits_{\cof(\xi)} .
\] 
However, this would be an overdoing: as shown in the next proposition, the inclusion in Corollary~\ref{cor:doublelimit} is in most cases proper, and thus the resulting class would be larger than expected.

\begin{proposition}\label{prop:double_lim_counterex}
Let \( \l \) be regular, \( X , Y \in \metr \), and let \( \xi < \l^+ \) be a limit ordinal.
Suppose that \( A \in \Sii{0}{\xi+1}(X) \setminus \Pii{0}{\xi+1}(X) \) and that there are at least two distinct points \( y_0,y_1 \in Y \). Let \( f \colon X \to Y \) be the ``characteristic function'' of \( A \) defined by \( f(x) = y_0 \) if \( x \in A \), and \( f(x) = y_1 \) otherwise. Then \( f \notin \M_{\xi+1}(X,Y) \), yet
\(f \in \llim{\Fin_\l}\left( \llim{\cof(\xi)} \, \M_{<\xi}(X,Y) \right)   \). In particular,
\[
\M_{\xi+1}(X,Y)  \subsetneq  \llim{\Fin_\l} \left( \llim{\cof(\xi)} \M_{<\xi}(X,Y) \right).
\]
\end{proposition}

\begin{proof}
We first prove a claim which might be of independent interest.

\begin{claim}\label{claim:locallyconstant}
Let \( A_0, \dotsc, A_n \in \Pii{0}{\xi}(X)\), and let \(g \colon X \to Y \) be such that \(g \restriction (X \setminus \bigcup_{j\leq n} A_j)\) and \(g \restriction A_j \), for \(j\leq n\), are all constant.
Then \(g = \lim_{\a<\cof(\xi)}g_\a\) for some sequence of function \( (g_\a)_{\a < \cof(\xi) }\) with \(g_\a \in \M_{<\xi}(X,Y)\).
\end{claim}

\begin{proof}[Proof of the Claim] 
Since $g$ is constant on each \(A_j\), we can assume without loss of generality that \(A_j \cap A_i=\emptyset\) for every \(j < i  \leq n\). Let \( \bar y \) be the value of \( g \) on \( X \setminus \bigcup_{j\leq n} A_j \), and \( y_j \) be the value of \( g \) on \( A_j \), for every \( j \leq n \).
Let \((\xi_\a)_{\a<\cof(\xi)}\) be a strictly increasing sequence of ordinals cofinal in \(\xi\). 
By Lemma~\ref{lem:reg}\ref{lem:reg2}, each \( X \setminus A_j \) can be written as \(X \setminus A_j=\bigcup_{\a<\cof(\xi)}B^j_\a\) where \( B^j_{\a} \subseteq B^j_{\a'} \) if \( \a \leq \a' < \cof(\xi) \) and  \(B^j_\a \in \Dee{0}{\xi_{\a+2}}(X)\). 
For every \(\a<\cof(\xi)\), let \( g_\a \colon X \to Y \) be defined by
\[
g_\a(x)= \begin{cases}
y_j, & \text{if } x \notin \bigcap_{j \leq n} B^j_\a \text{ and } j\leq n \text{ is smallest such that } x \notin B^j_\a, \\
\bar y , & \text{ otherwise}.
\end{cases}
\]
Each \(g_\a\) is \(\Sii{0}{\xi_\a+2}\)-measurable because it is locally constant on a \(\Dee{0}{\xi_\a+2}\)-partition,
and hence it belongs to \( \M_{< \xi}(X,Y) \). 
Pick any \( x \in X \). If \( x \notin \bigcup_{j \leq n} A_j  \), then for every \( j \leq n \) there is \( \a_j < \cof(\xi) \) such that \( x \in B^j_{\a_j} \). Set \( \bar \a = \max \{ \a_j \mid j \leq n \} \): then for every \( \bar \a \leq \a < \cof(\xi) \) we have \( x \in \bigcap_{j \leq n} B^j_\a \), and hence \( g_\a(x) = \bar y = g(x) \).
If instead \( x \in \bigcup_{j \leq n} A_j \), then there is a unique \( \bar \jmath \leq n \) such that \( x \in A_{\bar \jmath} \) because the sets \( A_j \) are assumed to be pairwise disjoint. For each \( j < \bar \jmath \), let \( \a_j < \cof(\xi) \) be such that \( x \in B^j_{\a_j} \), and set \( \bar \a = \max \{ \a_j \mid j < \bar \jmath \}\). Then by construction \( g_\a(x) = y_j = g(x) \) for every \( \bar \a \leq \a < \cof(\xi) \).
This shows that \(g=\lim_{\a<\cof(\xi)}g_\a\) and concludes our proof. 
\end{proof}

By Claim~\ref{claim:locallyconstant}, it suffices to show that \( f = \lim_{d \in \Fin_\l} f_d \), where each \( f_d \colon X \to Y \) is locally constant on both \(  A_d \) and \( X \setminus A_d \) for some \( A_d \in \Pii{0}{\xi}(X) \).
To this aim, let \( \{ A_\a \mid \a < \l \} \subseteq \Pii{0}{\xi}(X) \) be such that \( A = \bigcup_{\a < \l} A_\a\). 
For each \( d \in \Fin_\l  \), let \( A_d = \bigcup \{ A_\a \mid \a \in d \} \), and then let \( f_d(x) = y_0 \) if \( a \in A_d \) and \( f_d(x) = y_1 \) otherwise. 
If \( x \notin A \), then \( f_d(x) = y_1 = f(x) \) for all \( d \in \Fin_\l \). If instead \( x \in A \), then \( f_d(x) = y_0 = f(x) \) for all \( d \supseteq \{ \a \} \), where \( \a < \l \) is any ordinal such that \(x \in A_\a \). 
This shows that \( f = \lim_{d \in \Fin_\l} f_d \), as desired.
\end{proof}

The core problem with using the double limit \( \lim_{\Fin_\l} \circ \lim_{\cof(\xi)} \) at limit levels is actually the fact that even when \( \cof(\xi) < \l \), the class \( \M_\xi(X,Y) \) is not closed under \( \cof(\xi) \)-limits. More in detail, let \( \l \) be regular and \( X,Y \in \metr \). Then for every \( 1 \leq \xi < \l^+  \), the class \( \Dee{0}{\xi+1}(X) \) is a \( \l \)-algebra. It easily follows that if \( D \) is a directed set of size smaller than \( \l \) and \( f \in \Dlim \M_\xi(X,Y)\), then  $f$ is $\Dee{0}{\xi+1}$-measurable. However, this is optimal. To see this, suppose that there is \( A \in \Pii{0}{\xi}(X) \setminus \Sii{0}{\xi}(X) \) and that \( Y \) contains at least two distinct points \( y_0 \) and \( y_1 \). Let \( f \colon X \to Y \) be defined by \( f(x) = y_0 \) if \( x \in A \) and \( f(x) = y_1 \) otherwise. Then by Claim~\ref{claim:locallyconstant} the function \( f \) is a \( \cof(\xi) \)-limit of functions in \( \M_{<\xi}(X,Y) \subseteq \M_\xi(X,Y) \), yet it is not \( \Sii{0}{\xi} \)-measurable by the choice of \( A \).

To overcome this difficulty,
we have to perform a deeper analysis of limits over directed sets of small size.
Let \( (D, \leq)\) be a directed set. A covering \( (X_d)_{d \in D} \) of a set \( X \) is (\markdef{\( D \)-})\markdef{increasing} if \( X_d \subseteq X_{d'} \) for every \( d,d' \in D \) such that \( d \leq d' \). 

\begin{definition}\label{def_hat_limit}
Let \( (D, \leq)\) be a directed set, and let
\(f\) and \( (f_d)_{d \in D} \) be functions between topological spaces \( X \) and \( Y \). Then \(f=\widehat{\lim}_{d \in D}f_d\) if there exists a \( D \)-increasing covering $(X_d)_{d\in D}$ of $X$ such that $f_{d'} \restriction X_d =f \restriction X_d$ for every \( d,d' \in D \) with $d' \geq d$.
\end{definition}

Notice that if \( f=\widehat{\lim}_{d \in D}f_d \), then there is a \emph{canonical} increasing covering \( (X_d)_{d \in D} \) of \( X \) witnessing that, namely,
\begin{equation} \label{eq:canonical}
X_d = \{ x \in X \mid f_{d'}(x) = f(x) \text{ for all }d' \geq d\} .
\end{equation}

Obviously, \(f=\widehat{\lim}_{d \in D}f_d\) implies \(f=\lim_{d \in D}f_d\). The reverse is not true in general. For example, assume that \( \l \) is regular and consider \( D = \l \), ordered as usual. Let \( f \colon \pre{\l}{\l} \to \pre{\l}{\l} \) be constant with value \( 0^{(\l)} \) and, for every \( \a < \l \), let \( f_\a \colon \pre{\l}{\l} \to \pre{\l}{\l} \) be constant with value \( 0^{(\a)} \conc 1^{(\l)} \). Then \( f = \lim_{\a < \l} f_\a \), but \( f \neq \widehat{\lim}_{\a < \l} f_\a \) for the simple reason that \( f(x) \neq f_\a(x) \) for every \( \a < \l \) and \( x \in \pre{\l}{\l}\). 

In contrast, we are now going to show that if \( D \) has small size, then the two notions of limit coincide. Let \( \D_{<\l} \) be the collection of all directed sets of size strictly smaller than \( \l \).\label{D<lambda} Clearly, \( \D_{< \l} \) is a proper subclass of \( \D_\l \).

\begin{lemma}\label{short_lim_ev_const} 
Let \( \l \) be regular, \( X , Y \in \metr \), and let  \(D \in \D_{< \l} \).
Let \( (f_d)_{d \in D} \) be a family of functions from \( X \) to \( Y \).
Then for every \( f \colon X \to Y \), the following are equivalent:
\begin{enumerate-(1)}
\item \label{short_lim_ev_const-1}
\( f  = \lim_{d \in D} f_d \);
\item \label{short_lim_ev_const-2}
for every $x \in X$ there exists $d \in D$ such that $f_{d'}(x)=f(x)$ for every $d' \geq d$;
\item \label{short_lim_ev_const-3}
\( f  = \widehat{\lim}_{d \in D} f_d \).
\end{enumerate-(1)} 
\end{lemma}

\begin{proof} 
By Theorem~\ref{theorem:sikorski}, without loss of generality we can assume that \( X,Y \subseteq \pre{\l}{\l} \).

\ref{short_lim_ev_const-1} \( \Rightarrow \) \ref{short_lim_ev_const-2}
Fix $x \in X$. For every $\b <\l$ there exists $d_\b \in D$ such that $f_d(x)\restriction\b = f(x)\restriction\b$, for all $ d \geq d_\b$. 
For every $d \in D$, let $G_d=\{ \b<\l \mid d_\b = d \}$. Since $(G_d)_{d \in D}$ is a covering of $\l$ in \( |D| \)-many pieces and \( \l \) is regular, there is $\bar d \in D $ such that $|G_{\bar d}|=\l$, so that \( G_{\bar d} \) is unbounded in \( \l \). 
It follows that for every $ d' \geq \bar d$ there are arbitrarily large $\b<\l$ such that $f_{d'}(x)\restriction\b = f(x)\restriction\b $, and thus $f_{d'}(x)=f(x)$.  

\ref{short_lim_ev_const-2} \( \Rightarrow \) \ref{short_lim_ev_const-3}
For \(d \in D \), let \( X_d \) be as in equation~\eqref{eq:canonical}.
By definition \( X_d \subseteq X_{d'} \) whenever \( d \leq d' \), and~\ref{short_lim_ev_const-2} ensures that \( \bigcup_{d \in D} X_d = X \).
It is then easy to verify that \( (X_d)_{d \in D} \) witnesses \( f  = \widehat{\lim}_{d \in D} f_d \).

\ref{short_lim_ev_const-3} \( \Rightarrow \) \ref{short_lim_ev_const-1}
Obvious.
\end{proof}

Although they are equivalent when \( D \in \D_{< \l} \), the advantage of using the limit operator \( \widehat{\lim}_{d \in D} \) instead of \( \lim_{d \in D} \) is that we can naturally impose a definability condition on the increasing coverings involved in its definition. Lemma~\ref{short_limits_xi_bis} and Theorem~\ref{thm:measurable_limits_2} show that this move fixes the problems encountered when dealing with limit levels of cofinality smaller than \( \l \).

\begin{definition}\label{def_hat_limit_definable}
Let \( (D, \leq)\) be a directed set, \( 1 \leq \xi < \l^+ \) be an ordinal, and let
\(f\) and \( (f_d)_{d \in D} \) be functions between topological spaces \( X \) and \( Y \). Then \(f=\widehat{\lim}{}^\xi_{d \in D}f_d\) if there exists a \( D \)-increasing covering $(X_d)_{d\in D}$ of $X$ witnessing \( f = \widehat{\lim}_{d \in D} f_d \) such that 
\( X_d \in \Sii{0}{\xi}(X) \) for every \( d \in D \).
\end{definition}

\begin{remark}
It can be shown that if \( f = \widehat \lim_{d \in D} f_d \) and \( f_d \in \M_\xi(X,Y)\) for every \( d \in D \), then the canonical increasing covering \( (X_d)_{d \in D} \) from equation~\eqref{eq:canonical} is such that \( X_d \in \Pii{0}{\xi}(X) \) for all \( d \in D \), because \( X_d = \bigcap_{d',d'' \geq d} \{ x \in X \mid f_{d'}(x) = f_{d''}(x) \}\). Thus in Definition~\ref{def_hat_limit_definable} we are adopting the weakest possible definability requirement.
\end{remark}

In contrast to the other results of this section, which all require \( \l \) to be regular, the next two results hold unconditionally, and indeed they will be used in Section~\ref{subsec:singular} too. 

\begin{lemma}\label{short_limits_xi_bis}
Let \( X , Y \in \metr \) and \( D \in \D_\l \). For every \( 1 \leq \xi < \l^+ \), if \(f \in \llimhatdef{D}{\xi} \M_\xi(X,Y)\), then  \(f \in \M_{\xi}(X,Y) \).
\end{lemma}

\begin{proof}
Let \( f = \widehat{\lim}{}^\xi_{d \in D}f_d \), as witnessed by the \( D \)-increasing covering \( (X_d)_{d \in D} \). Then for every open set \( U \subseteq Y \) we have
\[
f^{-1}(U) = \bigcup_{d \in D} (f_d^{-1}(U) \cap X_d ),
\]
so that \( f^{-1}(U) \in \Sii{0}{\xi}(X) \) because \( X_d \in \Sii{0}{\xi}(X) \).
\end{proof}

The proof of the next result is essentially the same of Corollary~\ref{cor:doublelimit}: we only have to check that, thanks to Lemma~\ref{short_limits_xi_bis}, the definability condition added in Definition~\ref{def_hat_limit_definable} allows us to obtain an equality instead of an inclusion.

\begin{theorem}\label{thm:measurable_limits_2}
 Let \( X , Y \in \metr \), and further assume that \( \dim(X) = \dim(Y) = 0 \) if \( \mu = \o \).  
For every limit ordinal  \( 1 \leq \xi < \l^+ \),
\[
\M_{\xi+1}(X,Y) = \llim{\Fin_\l} \left( \llimhatdef{\cof(\xi)}{\xi} \M_{<\xi}(X,Y) \right).
\]
\end{theorem}

\begin{proof}
Without loss of generality, we can assume \(X,Y \subseteq \pre{\mu}{\l}\) by Theorems~\ref{theorem:sikorski} and~\ref{theorem:sikorski2}.
The inclusion from left to right follows again from Proposition~\ref{prop:finlimit} and Lemma~\ref{lem:locallyconstant}, together with Remark~\ref{rem:partition}. 

For the reverse inclusion, observe that \( \llimhatdef{\cof(\xi)}{\xi} \M_{<\xi}(X,Y) \subseteq \M_\xi(X,Y) \) by Lemma~\ref{short_limits_xi_bis}, so it is enough to check that \( \llim{\Fin_\l} \M_\xi(X,Y) \subseteq \M_{\xi+1}(X,Y) \). 
Since \( |\Fin_\l| = \l \) this can easily be obtain from the same argument used in the implication \ref{teo:D_limits4} $\Rightarrow$ \ref{teo:D_limits1} of Theorem~\ref{teo:D_limits}, the only difference being that now the functions \( f_d \) are assumed to be \( \Sii{0}{\xi} \)-measurable.
\end{proof}

We are now ready to discuss some simple and natural generalizations to the uncountable setting of the Baire stratification of Borel functions, with the goal of obtaining an analogue of Theorem~\ref{thm:Bairestratification-classical}. 
Recall that in the classical case \( \l = \o \), the class of Baire class \( \xi \) functions is defined by taking continuous functions if \( \xi = 0 \), and (\( \omega \)-)limits of functions of lower Baire class if \( \xi > 0 \).
As discussed at length, when moving to a regular cardinal we can naturally use \( \l \)-limits at successor stages and at limit levels with cofinality \( \l \), while in the remaining cases we have to use a different limit operator. This leads us to the following first proposal.
To simplify the notation, we write \( \mathcal{B}_{<\xi}(X,Y) \) instead of \( \bigcup_{\xi' < \xi} \mathcal{B}_{\xi'}(X,Y) \).

\begin{definition}\label{baire_class}
Let \( \l \) be regular, and let \( X , Y \in \metr \). For every \( \xi < \l^+ \), we recursively define the collection \( \mathcal{B}_\xi(X,Y) \) of \markdef{\( \l \)-Baire class \( \xi \) functions} as follows:
\begin{enumerate-(1)}
\item \label{baire_class-1}
\( \mathcal{B}_0(X,Y) \) is the collection of all continuous functions \( f \colon X \to Y \).
\item \label{baire_class-2}
If \( \xi \) is either a successor ordinal or a limit ordinal with \( \cof(\xi) = \l \), then 
\[  
\mathcal{B}_\xi(X,Y) = \llim{\l} \mathcal{B}_{<\xi}(X,Y).
\]
\item \label{baire_class-3}
If \( \xi \) is a limit ordinal with \(  \cof(\xi) < \l \), then 
\[
\mathcal{B}_\xi(X,Y) = \llim{\Fin_\l} \left( \llimhatdef{\cof(\xi)}{\xi} \mathcal{B}_{<\xi}(X,Y) \right).
\]
\end{enumerate-(1)}
\end{definition}

Definition~\ref{baire_class} is designed to be as close as possible to the classical definition and to minimize the kind of limit operators involved, at the price of having to consider different cases depending on the ordinal \( \xi < \l^+ \). 
In the opposite direction, one could desire to have a uniform definition which maximizes the class of limit operators that can be employed. This leads us to the idea of allowing the use of the special double-limits from part~\ref{baire_class-3} of Definition~\ref{baire_class} at every level of the hierarchy, and to replace \( \Fin_\l \) and \( \cof(\xi) \) with any pair of directed sets in \( \D_\l \).
We again write \( \widehat{\mathcal{B}}_{< \xi}(X,Y) \) instead of \( \bigcup_{\xi' < \xi} \widehat{\mathcal{B}}_{\xi'}(X,Y) \). 

\begin{definition}\label{baire_class_hat}
Let \( \l \) be regular, and let \( X , Y \in \metr \). For every \( \xi < \l^+ \), we recursively define the classes of functions \( \widehat{\mathcal{B}}_\xi(X,Y) \) as follows: 
\begin{enumerate-(1)}
\item \label{baire_class_hat-1}
\( \widehat{\mathcal{B}}_0(X,Y) \) is the collection of all continuous functions \( f \colon X \to Y \).
\item \label{baire_class_hat-2}
If \( \xi > 0 \), then
\[  
\widehat{\mathcal{B}}_\xi(X,Y) = \llim{\D_\l} \left( \llimhatdef{\D_{\l}}{\xi} \widehat{\mathcal{B}}_{<\xi}(X,Y) \right).
\]
\end{enumerate-(1)}
\end{definition}

We are now going to show that for most spaces \( X \) and \( Y \), both definitions do the job, and they are in fact equivalent to each other.
We remark that the two options are the extremes of a whole range of intermediate possibilities, which however would give rise to the very same classes of functions because of the next theorem.

\begin{theorem}\label{baire_class_xi_teo} 
Let \( \l \) be regular, \( X , Y \in \metr \), and suppose that \( Y \) is spherically complete.  
Then for every $ \xi<\l^+$,
\[
\mathcal{B}_\xi(X,Y) = \widehat{\mathcal{B}}_\xi(X,Y) = \M_{\xi+1}(X,Y).
\]
\end{theorem}

\begin{proof}
By induction on \( \xi < \l^+ \). The equality \( \mathcal{B}_\xi(X,Y)= \M_{\xi+1}(X,Y) \) follows from Theorems~\ref{thm:measurable_limits} and~\ref{thm:measurable_limits_2}, while the inclusion \(\mathcal{B}_\xi(X,Y) \subseteq  \widehat{\mathcal{B}}_\xi(X,Y) \) is obvious. Finally, to prove 
\( \widehat{\mathcal{B}}_\xi(X,Y) \subseteq \M_{\xi+1}(X,Y) \) it is enough to use Lemma~\ref{short_limits_xi_bis} and the usual computation yielding \( \Dlim \M_\xi(X,Y) \subseteq \M_{\xi+1}(X,Y) \) for every \( D \in \D_\l \).
\end{proof}

The assumption that \( Y \) be spherically complete in Theorem~\ref{baire_class_xi_teo} is very mild, as every \( \l \)-Cauchy-complete%
\footnote{These spaces are one of the natural generalizations to the uncountable setting of Polish spaces, and they are indeed called \( \GG \)-Polish spaces in the specialized literature (see e.g.~\cite{AMRS23}).}
\( X \in \metr \) can easily be turned into a spherically complete space by adding at most \( \l \)-many points. 
However, if one wants to dispense from such assumption and work with arbitrary spaces in \( \metr \), it is enough to modify Definition~\ref{baire_class} as follows:
\begin{itemizenew}
\item 
when \( \xi \) is a successor ordinal, \( \l \)-limits have to be replaced with \( \Fin_\l \)-limits;
\item
at all limit levels \( \xi \) we use the double limit \( \lim_{\Fin_\l} \circ \ \widehat{\lim}^\xi_{\cof(\xi)} \), independently of the cofinality of \( \xi \).
\end{itemizenew} 
This works because even after such modification, in the first part of the proof of Theorem~\ref{baire_class_xi_teo} we can still prove the inclusion \( \M_{\xi+1}(X,Y) \subseteq \B_\xi(X,Y) \) by using Proposition~\ref{prop:finlimit} and (a suitable variant of) Corollary~\ref{cor:doublelimit}.
This alternative approach turns out to work also in the singular case, as briefly discussed in the next section.

\subsection{The singular case}\label{subsec:singular}

We now move to the case where \( \mu = \cof(\l) < \l \). There are two major differences from the regular case:
\begin{enumerate-(1)}
\item 
First, \( \l \)-limits are no longer relevant, as \( \l \)-limits and \( \mu \)-limits are obviously equivalent.
\item 
Secondly, when \( \xi \) is limit we do not have to distinguish cases depending on the cofinality of \( \xi \): since \( \l \) is singular, \( \cof(\xi) < \l \) for every limit \( \xi < \l^+ \).
\end{enumerate-(1)}

This naturally leads us to the following variation of Definition~\ref{baire_class}.

\begin{definition}\label{baire_class_singular}
Let \( \l \) be singular, and let \( X , Y \in \metr \). For every \( \xi < \l^+ \), we recursively define the collection \( \mathcal{B}_\xi(X,Y) \) of \markdef{\( \l \)-Baire class \( \xi \) functions} as follows:
\begin{enumerate-(1)}
\item \label{baire_class-1}
\( \mathcal{B}_0(X,Y) \) is the collection of all continuous functions \( f \colon X \to Y \).
\item \label{baire_class-2}
If \( \xi \) is a successor ordinal, then 
\[  
\mathcal{B}_\xi(X,Y) = \llim{\Fin_\l} \mathcal{B}_{<\xi}(X,Y).
\]
\item \label{baire_class-3}
If \( \xi \) is a limit ordinal, then 
\[
\mathcal{B}_\xi(X,Y) = \llim{\Fin_\l} \left( \llimhatdef{\cof(\xi)}{\xi} \mathcal{B}_{<\xi}(X,Y) \right).
\]
\end{enumerate-(1)}
\end{definition}

Moreover, it still makes sense to consider the alternative approach taken in Definition~\ref{baire_class_hat}: in this case, the definition of the classes \( \widehat{\mathcal{B}}_\xi(X,Y) \) needs not to be updated. As in the regular case, it turns out that both definitions allow us to generalize Theorem~\ref{thm:Bairestratification-classical}.

\begin{theorem}\label{baire_class_xi_teo_sing} 
Let \( \l \) be singular, \( X , Y \in \metr \), and further assume that \( \dim(X) = \dim(Y) = 0 \) if \( \mu = \o \). 
Then for every $ \xi<\l^+$,
\[
\mathcal{B}_\xi(X,Y) = \widehat{\mathcal{B}}_\xi(X,Y) = \M_{\xi+1}(X,Y).
\]
\end{theorem}

\begin{proof}
We argue as in the proof of Theorem~\ref{baire_class_xi_teo}, with the following exception. Instead of showing that \( \mathcal{B}_\xi(X,Y)= \M_{\xi+1}(X,Y) \), we just prove the inclusion \( \M_{\xi+1}(X,Y) \subseteq \mathcal{B}_\xi(X,Y) \) using Proposition~\ref{prop:finlimit} (in the successor case) and a suitable variant of Corollary~\ref{cor:doublelimit} (in the limit case) instead of Theorems~\ref{thm:measurable_limits} and~\ref{thm:measurable_limits_2}. The rest of the proof goes unchanged.
\end{proof}


\section{Uniform limits} \label{sec:uniform_lim}

This short section aims at generalizing the results from~\cite[Section 3]{MR09} to the uncountable setting. The theme is that of characterizing \( \Sii{0}{\xi} \)-measurable functions in terms of \emph{uniform} limits of simpler functions. Although somewhat unrelated to the previous sections, these results add more information on the structure of \( \Sii{0}{\xi} \)-measurable functions with respect to another kind of limit, and thus we feel that they fit well with the general topic of the paper.

Recall that in Section~\ref{subsec:metr} we fixed a field \( \GG \) with degree \( \mu = \cof(\l) \) and a strictly decreasing sequence \( (r_\a)_{\a < \mu} \) coinitial in \( \GG^+ \). 
By the choice we made, \( r_{\a+1} + r_{\a+1} \leq r_\a \) for every \( \a < \mu \).
Let \( \metrhat\) be the collection of all \( \GG \)-metric spaces of weight at most \( \l \).
Given a directed set \( D \), a topological space \( X \in \metr \), and a \( \GG \)-metric space \( (Y,d_Y) \in \metrhat \), we say that a function \( f \colon X \to Y \) is the \markdef{uniform \( D \)-limit} of a family \( (f_d)_{d \in D} \) of functions from \( X \) to \( Y \), and we write \( f = \ulim_{d\in D} f_d\), if for every \(\epsilon \in \GG^+\) there is \( d \in D \) such that \(d_Y(f_{d'}(x), f(x))<\epsilon\) for  every \( x \in X \) and \( d' \geq d \). It is not hard to see that one could restrict the attention to directed sets of size at most \( \mu \): if \( f = \ulim_{d \in D} f_d \), then there is \( D' \subseteq D\) with \( |D'| = \min \{ \mu, |D| \} \) such that \( f  = \ulim_{d \in D'} f_d\).

As in the classical setting, it is easy to check that continuous functions are closed under uniform \( \l \)-limits. 
We generalize this to higher levels and to all directed sets \( D \).

\begin{lemma}\label{lem:closure_uniformnew}
Let \( D \) be any directed set, \( X \in \metr \), \( (Y,d_Y) \in \metrhat \), and \( 1 \leq \xi<\l^+\).
Then \( \M_\xi(X,Y) \) is closed under uniform \( D \)-limits. 
\end{lemma}

\begin{proof}
Suppose that \( f = \ulim_{d \in D } f_d \) with \( f_d \in \M_\xi(X,Y) \) for every \( d \in D \).
We need to show that given any \( y \in Y \) and \( \varepsilon \in \GG^+ \), the preimage of the open ball \( B_{d_Y}(y,\varepsilon)\) belongs to \( \Sii{0}{\xi}(X) \). Let \( \b < \mu \) be such that \( r_\b < \varepsilon \), and for each \( \a < \mu \) set \( \varepsilon_\a = \varepsilon - r_{\b+\a} \). Then \( 0 < \varepsilon_\a < \varepsilon \), and for every \( \varepsilon' < \varepsilon \) there is \( \a < \mu \) such that \( \varepsilon' \leq \varepsilon_\a \).

By choice of \( (f_d)_{d \in D} \), for every \( \a < \mu \) there is \( d_\a \in D \) such that \( d_Y(f_d(x),f(x)) < r_{\b+\a} \) for every \( x \in X \) and \( d \geq d_\a \). We claim that
\[  
f^{-1}(B_{d_Y}(y,\varepsilon)) = \bigcup_{\a < \mu}  
f^{-1}_{d_\a}(B_{d_Y}(y,\varepsilon_\a)),
\]
from which the result clearly follows.

For the inclusion from right to left, fix \( \a< \mu \) and a point \( x \in f^{-1}_{d_\a}(B_{d_Y}(y,\varepsilon_\a)) \). By the triangle inequality and the choice of \( d_\a \),
\[  
d_Y(y,f(x)) \leq d_Y(y,f_{d_\a}(x)) + d_Y(f_{d_\a}(x),f(x)) < \varepsilon_\a + r_{\b+\a} = \varepsilon,
\]
so that \( x \in f^{-1}(B_{d_Y}(y,\varepsilon)) \). 

Conversely, given \( x \in f^{-1}(B_{d_Y}(y,\varepsilon)) \) let \( \varepsilon' = d_Y(y,f(x)) < \varepsilon \). Let \( \a < \mu \) be such that \( r_{\b+\a} < \varepsilon - \varepsilon'\), so that \(  \varepsilon' + r_{\b+\a+1} + r_{\b+\a+1} \leq \varepsilon' + r_{\b+\a} \leq \varepsilon \) and hence \( \varepsilon' + r_{\b+\a+1} \leq \varepsilon - r_{\b+\a+1} = \varepsilon_{a+1} \). Then by the triangular inequality and the choice of \( d_\a\)
\[
d_Y(y,f_{d_{\a+1}}(x)) \leq d_Y(y,f(x)) + d_Y(f(x),f_{d_{\a+1}}(x)) < \varepsilon' + r_{\b+\a+1} \leq \varepsilon_{\a+1}.
\]
Therefore \( x \in f^{-1}_{d_{\a+1}}(B_{d_Y}(y,\varepsilon_{\a+1}))  \) and we are done.
\end{proof}

Let \( X \) and \( Y \) be topological spaces, and let \( 1 \leq \xi < \l^+ \). A \markdef{\(\Dee{0}{\xi}\)-function} is a map \( f \colon X \to Y \) such that \(f^{-1}(A) \in \Sii{0}{\xi}(X)\), for every \(A \in \Sii{0}{\xi}(Y)\). The set of \( \Dee{0}{\xi} \)-functions is denoted by \( \boldsymbol\Delta_{\xi}(X, Y) \).
It is easy to check that \( \Dee{0}{\xi} \)-functions provide an alternative stratification of \( \l^+ \)-Borel measurable functions: indeed, \( \boldsymbol{\Delta}_\xi(X,Y) \subseteq \M_\xi(X,Y) \subseteq \boldsymbol{\Delta}_{\xi \cdot \omega}(X,Y) \) for every \( 1 \leq \xi < \l^+ \).
Moreover, under mild assumptions on the spaces \( X \) and \( Y \) the set \( \boldsymbol\Delta_{\xi}(X, Y) \) is a proper subclass of  \(\M_{\xi}(X,Y)\).
Notice also that if a function is locally in \( \boldsymbol\Delta_{\xi}(X, Y) \) on a \( \Sii{0}{\xi} \)-partition of \( X \) of size at most \( \l \), then it is a \( \Dee{0}{\xi} \)-function itself. 
We now have all the tools to mimic the proof of~\cite[Theorem 3.5]{MR09}  and get the following characterization of \( \M_{\xi+1}(X,Y) \) in terms of uniform \( \l \)-limits of simpler functions.

\begin{theorem}\label{thrm:delta_funnew} 
Let $X \in \metr$, \( (Y,d_Y) \in \metrhat \), and $1 \leq \xi<\l^+$. For every function $f \colon X \to Y$, the following are equivalent:
\begin{enumerate-(1)}
\item \label{delta_funnew1} 
$f \in \M_{\xi}(X,Y) \);
\item \label{delta_funnew2} 
\( f = \ulim_{\a < \l} f_\a\), where each \( f_\a \) is locally constant on a \( \Sii{0}{\xi} \)-partition of \( X \) of size at most \( \l \);
\item \label{delta_funnew3} 
\( f = \ulim_{\a < \l} f_\a \) with \( f_\a \in \boldsymbol\Delta_{\xi}(X,Y) \) for all \( \a < \l \).
\end{enumerate-(1)}
Moreover, in part~\ref{delta_funnew2} we can replace locally constant functions with locally Lipschitz or even just locally continuous maps.
\end{theorem}

\begin{proof} 
The implication \ref{delta_funnew2} \( \Rightarrow \) \ref{delta_funnew3} follows from the fact that every constant (or Lipshitz, or continuous) function belongs to \( \boldsymbol\Delta_{\xi}(X,Y) \). Moreover, if \ref{delta_funnew3} holds then \( f_\a \in \M_\xi(X,Y) \) for all \( \a < \l \), and thus \( f \in \M_{\xi}(X,Y) \) by Lemma~\ref{lem:closure_uniformnew}, so that \ref{delta_funnew1} is satisfied. It remains to prove \ref{delta_funnew1} \( \Rightarrow \) \ref{delta_funnew2}. 

Let \( \{ y_i \mid i < \l \} \) be a dense subset of \( Y \), and for each \( \a < \mu \) and \( i < \l \) let \( U^i_\a = B_{d_Y}(y_i,r_\a) \). Fix \( \a < \mu \). The family \( \{ U^i_\a \mid i < \l \} \) is an open covering of \( Y \).
Since \( f \in \M_{\xi}(X,Y) \), each set \( A^i_\a = f^{-1}(U^i_\a) \) belongs to \( \Sii{0}{\xi}(X) \). 
By the \( \l \)-reduction property (Theorem~\ref{22.16}), we can find a disjoint \( \Sii{0}{\xi} \)-covering \( \B_\a = \{ B^i_\a \mid i < \l \} \) of \( X \) such that \( B^i_\a \subseteq A^i_\a \) for every \( i < \l \). Thus the nonempty elements of \( \B_\a \) form a \( \Sii{0}{\xi} \)-partition of \( X \). Let \( f_\a \) be defined on an arbitrary \( x \in X \) by letting \( f(x) = y_i\), where \( i < \l \) is such that \( x \in B^i_\a \).
Then \( f_\a \) is locally constant on the above mentioned \( \Sii{0}{\xi} \)-partition, and \( f = \ulim_{\a < \l} f_\a \) by the choice of the open sets \( U^i_\a \).
\end{proof}

\begin{corollary}
Let $X \in \metr$, \( (Y,d_Y) \in \metrhat \), and $1 \leq \xi<\l^+$. Then
\[  
\M_{\xi}(X,Y) = \ullim{\D_\l} \boldsymbol\Delta_{\xi}(X,Y).
\]
\end{corollary}

\begin{proof}
The left-to-right inclusion follows from Theorem~\ref{thrm:delta_funnew}, while the reverse inclusion holds by \( \boldsymbol{\Delta}_{\xi}(X,Y) \subseteq \M_{\xi}(X,Y) \) and Lemma~\ref{lem:closure_uniformnew}.
\end{proof}

By virtue of Theorems~\ref{baire_class_xi_teo} and~\ref{baire_class_xi_teo_sing}, all the results of this section also hold for the collection \( \B_\xi(X,Y)\) of \( \l \)-Baire class \( \xi \) functions, for every \( \xi < \l^+ \) (under the appropriate hypotheses on \( X \) and \( Y \)).
In particular, while \( \B_\xi(X,Y) \) is far from being closed under \( \D_\l \)-limits, which indeed generate the next class \( \B_{\xi+1}(X,Y) \), it is instead closed under \emph{uniform} \( \D_\l \)-limits.

\section{Final remarks and open questions} \label{sec:questions}

Theorem~\ref{teo:D_limits}, Theorem~\ref{baire_class_xi_teo}, and Theorem~\ref{baire_class_xi_teo_sing} provide a quite satisfactory answer to our initial problem. However, as discussed at the end of Section~\ref{sec:Borelfunctions}, it would be interesting to understand if it is possible to further reduce the number of limits needed to perform this kind of analysis. In particular, in the regular case it is natural to wonder whether sequential limits (of any length) suffice.
An attempt in this direction was made in~\cite{phdNob}, where \( \l \)-Baire class \( \xi \) functions are defined as the collection of all \( \k \)-limits, for \( \k \leq \l \), of functions with \( \l \)-Baire class smaller than \( \xi \). 
The setup is that of a regular cardinal \( \lambda \) and of functions between some fixed subset \( X \) of \( \pre{\l}{\l}\) and the whole generalized Baire space \( Y = \pre{\l}{\l}\), thus this definition is precisely what we are considering in the present discussion. 
After providing a game characterization of the classes \( \M_{\xi+1}(X,Y) \), in~\cite[Theorem 4.12]{phdNob} it is claimed that all \( \Sii{0}{\xi+1} \)-measurable functions are of \( \l \)-Baire class \( \xi \) according to the previous definition. 
The proof is by induction on \( \xi < \l^+ \), but as observed it has to fail at the first limit level with small cofinality, i.e.\ at stage \( \xi = \o\): since the functions of finite \( \l \)-Baire class are those in \( \M_{< \o}(X,Y) \), any \( \k \)-limit of such functions must be in \( \M_{< \o}(X,Y) \) itself if \( \k = \l \) (Proposition~\ref{prop:long_limits_counterex}), or it has to be \( \Dee{0}{\xi+1} \)-measurable if \( \k < \l \) (because \( \Dee{0}{\xi+1} \) is a \( \l \)-algebra). 
In other words: when confining ourselves to sequential limits, in order to capture \( \Sii{0}{\xi+1} \)-measurable functions for \( \xi \) limit with \( \cof(\xi)< \l \) it is necessary to use (at least) double limits. 
Despite the failure of the na\"ive approach undertaken in~\cite{phdNob}, using nested sequential limits it might still be possible to answer affirmatively the following questions, which to the best of our knowledge are still open.

\begin{question}
Let \( \l \) be regular, and \( X,Y \in \metr \). Does the collection of all \( \l^+\)-Borel measurable functions from \( X \) to \( Y \) coincide with the closure of continuous functions under \( \k \)-limits, where \( \k \) varies among regular cardinals up to \( \l \)? If \( \xi < \l^+ \) is limit with \( \cof(\xi) < \l \), is it true that \( \M_{\xi+1}(X,Y) =  \llim{\l} \left( \llimhatdef{\cof(\xi)}{\xi} \M_{<\xi}(X,Y) \right) \)? Can we at least have \( \M_{\xi+1}(X,Y) \subseteq  \llim{\l} \left( \llim{\cof(\xi)} \M_{<\xi}(X,Y) \right) \)?
\end{question}

A positive answer to these questions would provide optimal results. 
As a partial contribution in this direction, we notice the following.

\begin{proposition} \label{prop:partialresult}
Let \( X,Y \in \metr \), and let \( \xi < \l^+ \) be a limit ordinal with \( \cof(\xi) < \l \).
\begin{enumerate-(1)}
\item \label{prop:partialresult-1}
If \( f \colon X \to Y \) is locally constant on a \(\Sii{0}{\xi} \)-partition of \( X \) of size at most \( \l \), then \( f \in \llimhatdef{\cof(\xi)}{\xi} \M_{<\xi}(X,Y) \).
\item \label{prop:partialresult-2}
If \( f \colon X \to Y \) is \( \Sii{0}{\xi} \)-measurable, then \( f \in \ullim{\l} \left( \llimhatdef{\cof(\xi)}{\xi} \M_{<\xi}(X,Y) \right)\).
\end{enumerate-(1)}
\end{proposition}

\begin{proof}
\ref{prop:partialresult-1} 
The proof is similar to that of Claim~\ref{claim:locallyconstant}. Let \( (A_j)_{j < J} \), for some \( J \leq \l \), be a \( \Sii{0}{\xi} \)-partition of \( X \) such that \( f \) is locally constant on it. By Lemma~\ref{lem:reg} (if \( \mu > \o \)) and~\cite[Proposition 4.2.1]{DMR} (if \( \mu = \o \)), there are a strictly increasing sequence \( (\xi_\a)_{\a < \cof(\xi)} \) cofinal in \( \xi \) and sets \( B^j_\a \in \Dee{0}{\xi_\a}(X) \) such that \( \a \leq \a' \Rightarrow  B^j_{\a} \subseteq B^j_{\a'} \) and \( A_j = \bigcup_{\a < \cof(\xi)} B^j_\a \), for every \( j < J \). Fix any \( \bar y \in Y  \) and define \( g_\a \colon X \to Y \) by setting
\[  
g_\a(x) = 
\begin{cases}
f(x) & \text{if } x \in \bigcup_{j<J} B^j_\a \\
\bar y & \text{otherwise}.
\end{cases}
\]
Then \( g_\a \) is \( \Sii{0}{\xi_\a+1}\)-measurable, and \( f = \widehat{\lim}^\xi_{\a < \cof(\xi)} g_a \), as witnessed by the increasing covering \( (X_\a)_{\a < \cof(\xi)} \) of \( X \) given by \( X_\a = \bigcup_{j < J} B^j_\a \).

\ref{prop:partialresult-2}
Use part~\ref{prop:partialresult-1} and Theorem~\ref{thrm:delta_funnew}.
\end{proof}

We also observe that by Theorem~\ref{thrm:delta_funnew}, to deal with the case of arbitrary functions in \( \M_{\xi+1}(X,Y) \) for \( \xi < \l ^+ \) limit with \( \cof(\xi) < \l \), it is enough to solve the problem for functions which are locally constant on a \( \Sii{0}{\xi+1} \)-partition of size at most \( \l \). Arguing as in Section~\ref{sec:long} (and using in particular Claim~\ref{claim:locallyconstant}), such functions can be obtained through a limit operator of the form \( \lim_\l \circ \lim_{\cof(\xi)}\) if the \( \Sii{0}{\xi+1} \)-partition is finite or, more generally, of size smaller than \( \cof(\xi) \). Unfortunately, although the remaining gap might seem small, we cannot yet close it.

In the case of a singular cardinal \( \l \), instead, it seems hard to conjecture that sequential limits can suffice, as only short limits are available because \( \l \)-limits are equivalent to \( \mu \)-limits, where as usual \( \mu = \cof(\l) < \l \). However, it makes sense to ask whether \( \Fin_\l \)-limits are enough to describe the whole structure of \( \l^+ \)-Borel measurable functions and its stratification given by \( \l \)-Baire class \( \xi \) functions. The question makes sense also in the regular case.

\begin{question}
Let \( X,Y \in \metr \).
Does the closure under \( \Fin_\l \)-limits of \( \M_1(X,Y) \) coincide with the collection of all \( \l^+ \)-Borel measurable functions? 
In particular, is it true that \( \M_{\xi+1}(X,Y) = \llim{\Fin_\l} \M_{< \xi}(X,Y) \) when \( \xi < \l^+ \) is a limit ordinal?
\end{question}

Finally, recall that in Theorem~\ref{teo:D_limitsunique} we proved that there is a single directed set in \( \D_\l \), namely \( \widehat{\Fin}_\l \), which can generate the whole class of \( \l^+ \)-Borel measurable functions, globally. This naturally raises the question of whether it can also give a level-by-level result. As usual, the answer is positive for successor levels and, if \( \l \) is regular, also for limit levels of cofinality \( \l \). However, the remaining cases are still unclear.

\begin{question}
Let \( X,Y \in \metr \), and let \( \xi < \l^+ \) be a limit ordinal.
Is it true that \( \M_{\xi+1}(X,Y) = \llim{\widehat{\Fin}_\l} \M_{< \xi}(X,Y) \)?
\end{question}

\end{document}